\documentclass[12pt,reqno]{article}

\usepackage{cite,epic,eepic,euscript,verbatim,amsmath,amssymb,amsthm,amsfonts,afterpage,float,bm,paralist}

\usepackage{cancel}
\usepackage{color}
\usepackage{fancyhdr}
\usepackage{graphics}              
\usepackage{hyperref}              
\usepackage{graphicx,epsf}
\usepackage{makeidx}
\usepackage{epsfig}

\allowdisplaybreaks
\newtheorem{theorem}{Theorem}[section]
\newtheorem{lemma}{Lemma}[section]

\newtheorem{corollary}{Corollary}[section]

\newtheorem{remark}{Remark}[section]

\numberwithin{equation}{section}
\numberwithin{figure}{section}

\newcommand{\mres}{\mathbin{\vrule height 1.6ex depth 0pt width
0.13ex\vrule height 0.13ex depth 0pt width 1.3ex}}

\def\XXint#1#2#3{{\setbox0=\hbox{$#1{#2#3}{\int}$}
\vcenter{\hbox{$#2#3$}}\kern-.51\wd0}}

\newcommand{\reff}[1]{{\rm (\ref{#1})}}

\newcommand{\calH}{{\mathcal H}}
\newcommand{\R}{\mathbb{R}}            
\newcommand{\bV}{\mathbf V}            
\newcommand{\bF}{\mathbf F}            
\newcommand{\ve}{\varepsilon}          

\newcommand{\cA}{ {\mathcal A}_0  }
\newcommand{\calA}{ {\mathcal A} }

\newcommand{\disp}{\displaystyle}

\newcommand{\qref}[1]{(\ref{#1})}

\newcommand{\vep}{\varepsilon}

\newcommand{\cH}{{\cal H}}

\newcommand{\rw}{ {\rm{w}} }
\newcommand{\rf}{ {\rm {f}} }
\newcommand{\rp}{ {\rm{p}} }

\newcommand{\Gnj}{{G_{n_j}}}
\newcommand{\Gnjk}{{G_{n_{j_k}}}}

\newcommand{\be}{\begin{eqnarray}}
\newcommand{\ee}{\end{eqnarray}}
\newcommand{\nn}{\nonumber}
\newcommand{\ben}{\begin{eqnarray*}}
\newcommand{\een}{\end{eqnarray*}}

\newcommand{\jl}[1]{\textcolor{red}{[[JL: #1]]}}
\newcommand{\sd}[1]{\textcolor{blue}{[[SD: #1]]}}
\newcommand{\tcb}[1]{\textcolor{blue}{#1}}

\begin{document}

\title{ 
Convergence of Phase-Field Free Energy and Boundary Force for Molecular Solvation} 


\author{
Shibin Dai\thanks{Department of Mathematical Sciences,
New Mexico State University, Las Cruces, NM 88003, USA. 
Email: sdai@nmsu.edu.} 
\and
Bo Li\thanks{Department of Mathematics and Quantitative Biology Graduate Program,
University of California, San Diego, 
9500 Gilman Drive, Mail code: 0112, La Jolla, CA 92093-0112,  USA. Email: bli@math.ucsd.edu.}
\and
Jianfeng Lu
\thanks{Department of Mathematics, Department of Physics, and Department of Chemistry, 
Duke University, Box 90320, Durham, NC 27708-0320, USA.  
Email: jianfeng@math.duke.edu.}
}

\date{June 14, 2016}

\maketitle

\begin{abstract}
We study a phase-field variational model for the solvaiton of charged molecules with 
an implicit solvent.  The solvation free-energy functional of all phase fields 
consists of the surface energy, solute excluded volume and solute-solvent van der Waals 
dispersion energy, and electrostatic free energy. 
The surface energy is defined by the van der Waals--Cahn--Hilliard functional 
with squared gradient and a double-well potential.  
The electrostatic part of free energy is defined through 
the electrostatic potential governed by the Poisson--Boltzmann equation 
in which the dielectric coefficient is defined through the underlying phase field. 
We prove the continuity of the electrostatics---its potential, free
energy, and dielectric boundary force---with respect to the
perturbation of dielectric boundary.  We also prove the
$\Gamma$-convergence of the phase-field free-energy functionals to
their sharp-interface limit, and the equivalence of the convergence of
total free energies to that of all individual parts of free energy.
We finally prove the convergence of phase-field forces to their
sharp-interface limit.  Such forces are defined as the negative first
variations of the free-energy functional; and arise from stress
tensors. In particular, we obtain the force convergence for the van
der Waals--Cahn--Hilliard functionals with minimal assumptions.

\bigskip


\end{abstract}


{\allowdisplaybreaks


\section{Introduction}
\label{s:introduction}

We study the convergence of a phase-field variational model to its
sharp-interface limit for the solvation of charged molecules.  In this
section, we present first the sharp-interface then the phase-field
models of molecular solvation.  We also describe our main results and
discuss their connections to existing studies. 
To ease the presentation, the quantities are only formally defined in this
  section; their precise definitions are given in Section~\ref{s:MainResults}.

\subsection{A Sharp-Interface Variational Model of Solvation}
\label{ss:SharpInterfaceModel}

We denote by $\Omega \subset \R^3$ 
the entire solvation region. It is divided into a solute (e.g., protein) region
$\Omega_{\rm p}$ (p for protein) that contains solute atoms located at 
$x_1, \dots, x_N$, and solvent region $\Omega_{\rm w}$ (w for water), separated by
a solute-solvent (e.g., protein-water) interface $\Gamma$. 
The solute atomic positions $x_1, \dots, x_N$ are given and fixed. 
A solute-solvent interface is treated as a dielectric boundary as it separates the low
dielectric solutes from high dielectric solvent. 
In a variational implicit-solvent model, an optimal solute-solvent interface is defined 
as to 
minimize the solvation free-energy functional of all the possible interfaces $\Gamma \subset \Omega$
that enclose $x_1,\dots, x_N$
\cite{DSM_PRL06, DSM_JCP06, 
Wang_VISMCFA_JCTC12, 
Zhou_VISMPB_JCTC14}: 
\begin{align}
\label{FGamma}
F[\Gamma] & = P_0 \mbox{Vol}\, (\Omega_{\rm p}) +  \gamma_0 \mbox{Area}\, (\Gamma) 
+ \rho_0  \int_{\Omega_{\rm w}} U(x) \, dx + F_{\rm ele}[\Gamma].  
\end{align}
The first term of $F[\Gamma]$ describes the work it takes to create the solute 
region $\Omega_{\rm p}$ in a solvent medium at hydrostatic pressure $P_0,$ 
where $\mbox{Vol}\,(\Omega_{\rm p})$ is the volume of $\Omega_{\rm p}.$ 
The second term is the solute-solvent interfacial energy, where $\gamma_0$ is an effective,
macroscopic surface tension. 
The third term, in which $\rho_0 $ is the constant bulk solvent density, 
is the solute-solvent interaction energy described by a potential $U$ that 
accounts for the solute-excluded volume and solute-solvent van der Waals attraction.   
The interaction potential $U$ is often given by
\[
U(x) = \sum_{i=1}^N U_{\text{LJ}}^{(i)}(|x-x_i|),
\]
where each
\[
U_{\text{LJ}}^{(i)}(r) = 4\varepsilon_i\left[\left(\frac{\sigma_i}{r}\right)^{12}
- \left(\frac{\sigma_i}{r}\right)^6 \right]
\]
is a Lennard-Jones potential with parameters $\varepsilon_i$ of energy and $\sigma_i$ of length. 

The last term is the electrostatic free energy. In the classical Poisson--Boltzmann theory, 
it is defined to be  
\cite{Li_SIMA09,
DavisMcCammon_ChemRev90,
SharpHonig_Rev90,
CDLM_JPCB08, 
AndelmanHandbook95, 
ZhouJCP94,
Zhou_VISMPB_JCTC14}
\begin{equation}
\label{FeleGamma}
F_{\rm ele}[\Gamma]  =  \int_\Omega \left[ -\frac{\ve_\Gamma }{2} |\nabla \psi_{\Gamma}|^2
+ \rho \psi_{\Gamma}  - \chi_{\Omega_{\rm w}}  B(\psi_{\Gamma} )   \right] dx,  
\end{equation}
where $\psi  = \psi_{\Gamma }$ is the electrostatic potential. It 
solves the boundary-value problem of the Poisson--Boltzmann equation
\cite{CDLM_JPCB08, 
AndelmanHandbook95, 
ZhouJCP94,
Zhou_VISMPB_JCTC14}
\begin{align}
\label{PBE}
& \nabla\cdot\varepsilon_\Gamma \nabla\psi - \chi_{\Omega_{\rm w}}  B'(\psi )
  = - \rho \qquad  \text{in } \Omega,
\\
\label{BC}
&\psi  = \psi_\infty \qquad  \mbox{on } \partial \Omega. 
\end{align}
Here, the dielectric coefficient $\ve_\Gamma$ (in the unit of vacuum permittivity) is defined by
$\ve_\Gamma(x) = \ve_{\rm p} $ if $x \in \Omega_{\rm p}$ and 
$\ve_\Gamma(x) = \ve_{\rm w} $ if $x \in \Omega_{\rm w}$, 
where $\ve_{\rm p}$ and $\ve_{\rm w}$ are the dielectric coefficients (relative permittivities)
of the solute and solvent regions, respectively. 
In general, $\ve_{\rm p} \approx 1$ and $\ve_{\rm w} \approx 80.$
The function $\rho: \Omega \to \R$ is the density of solute atomic charges.
It is an approximation of the point charges $\sum_{i=1}^N Q_i \delta_{x_i}$,
where $Q_i$ is the partial charge carried by the $i$th atom at $x_i$ and 
$\delta_{x_i}$ denotes the Dirac mass at $x_i$ $(1 \le i \le N).$
The function $\chi_A$ is the characteristic function of $A.$ 
The function $\psi_\infty: \partial \Omega \to \R$ is a given boundary value of $\psi_\Gamma.$ 
The term $B(\psi_{\Gamma})$ models the ionic effect and the function $B$ is given by 
\[
B(s) = k_{\rm B} T \sum_{j=1}^M c_j^\infty \left( e^{- q_j s/ (k_{\rm B} T) } - 1 \right),
\]
where $k_{\text{B}}$ is the Boltzmann constant and $T$ absolute temperature,
and $c_j^{\infty}$ and $q_j = z_j e$ are the bulk concentration and
charge for the $j$th ionic species, respectively, with $z_j$ the
valence and $e$ elementary charge. Note that $B'' > 0$ on $\R$; so $B$ is strictly convex. 
We assume there are $M$ species of ions in the solvent. 
Moreover, in the bulk, the charge neutrality is reached: 
$
\sum_{j=1}^M q_j c_j^\infty = 0.
$
This implies that $B'(0) = 0,$ and hence $B$ is also minimized at $0.$  


For a smooth dielectric boundary $\Gamma$, we denote by $\nu$ its unit normal
pointing from the solute region $\Omega_{\rm p}$ to the solvent region $\Omega_{\rm w}$. 
We define the normal component of the boundary force (per unit surface area) as the negative variation, 
$-\delta_\Gamma F[\Gamma]: \Gamma \to \R$, of the solvation free energy $F[\Gamma]$ (cf.~\reff{FGamma}). 
It is given by \cite{CDLM_JPCB08,Zhou_VISMPB_JCTC14,LiChengZhang_SIAP11, 
Luo_PCCP12, CXDMCL_JCTC09, ChengChengLi_Nonlinearity11,XiaoLuo_JCP2013} 
\begin{align}
\label{BoundaryForce}
-\delta_{\Gamma}F[\Gamma]  
&=  -P_0 -  2 \gamma_0  H  + \rho_0  U 
- \frac{1}{2} \left( \frac{1}{\ve_{\rm p} } - \frac{1}{\ve_{\rm w} } \right) 
\left( \ve_\Gamma  \frac{\partial \psi_\Gamma }{\partial \nu } \right)^2 
 \nonumber \\
&\quad 
- \frac{1}{2} ( \ve_{\rm w}  - \ve_{\rm p} )  
\left| \nabla_\Gamma \psi_\Gamma \right|^2 -  B (\psi_\Gamma) \qquad \mbox{on } \Gamma, 
\end{align}
where $H$ is the mean curvature, defined as the average of principal curvatures,   
positive if $\Omega_{\rm p}$ is convex,   
$\psi_\Gamma$ is electrostatic potential defined by \reff{PBE} and \reff{BC},  
and $\nabla_\Gamma = (I - \nu \otimes \nu ) \nabla $, with $I$ the identity
matrix, is the surface gradient along $\Gamma$. 

\subsection{A Phase-Field Variational Model of Solvation}
\label{ss:PhaseFieldModel}

To incorporate more detailed physical and chemical properties in the
solute-solvent interfacial region, such as the asymmetry of dielectric
environment, Li and Liu \cite{LiLiu_SIAP15}, and Sun {\it et al.}
\cite{Sun_PFVISM_JCP15} constructed and implemented a related
phase-field model for the solvation of charged molecules (cf.\ also
\cite{LiZhao_SIAP13,PhaseField_JCP13}).  In such a model, a phase
field $\phi: \Omega \to \R,$ a continuous function that takes values
close to $0$ and $1$ in $\Omega$ except in a thin transition layer, is
used to describe the solvation system. The solute and solvent regions
(or phases) are approximated by $\{ \phi \approx 1 \}$ and
$\{ \phi \approx 0 \}$, respectively, and the thin transition layer is
the diffuse solute-solvent interface.  Let $\xi > 0$ be a small
number.  The phase-field solvation free-energy functional of phase
fields $\phi: \Omega\to \R$ is
\cite{PhaseField_JCP13,LiZhao_SIAP13,Sun_PFVISM_JCP15, LiLiu_SIAP15}:
\begin{align}
\label{Fxiphi}
F_\xi [\phi]  & =  P_0 \int_\Omega \phi^2 \, d x 
+ \gamma_0 \int_\Omega  \left[ \frac{\xi }{2} |\nabla\phi|^2 + \frac{1}{\xi } 
W(\phi ) \right] dx + \rho_0  \int_\Omega ( \phi - 1 )^2  U \, dx
+ F_{\rm ele}[\phi], 
\end{align}
where 
\begin{align}
\label{Felephi} 
F_{\rm ele}[\phi] =
\int_\Omega \left[ -\frac{\ve (\phi) }{2} |\nabla \psi_{\phi}|^2 
+ \rho \psi_{\phi} - (\phi-1)^2 B(\psi_{\phi} ) \right] dx, 
\end{align}
and 
 $\psi = \psi_{\phi}$ solves the boundary-value problem of the phase-field Poisson--Boltzmann equation
\begin{align}
\label{PhaseFieldPBE}
& \nabla\cdot\varepsilon(\phi)\nabla\psi - (\phi-1)^2 B'(\psi )
  = - \rho \qquad  \text{in } \Omega,
\\
\label{PhaseFieldBC}
&\psi  = \psi_\infty \qquad  \mbox{on } \partial \Omega.    
\end{align}

All the four terms in \reff{Fxiphi} correspond to those in the sharp-interface
free-energy functional \reff{FGamma}. 
The second integral term, in which 
\begin{equation}
\label{W}
W(\phi) = 18 \phi^2 ( 1 - \phi )^2, 
\end{equation}
is the van der Waals--Cahn--Hilliard functional \cite{vdW1893,Rowlinson_vdWtrans79,CahnHilliard58}
(sometimes called the Allen--Cahn functional \cite{AllenCahn79})
that is known to $\Gamma$-converge to the area of solute-solvent interface as $\xi \to 0$ 
\cite{Modica_ARMA87,Sternberg_ARMA88}. 
The pre-factor $18$ is so chosen that 
\[
\int_0^1 \sqrt{2 W(t)} \, dt = 1. 
\]
In the last term of electrostatic free energy, 
the dielectric coefficient $\ve = \ve(\phi)$ is constructed to be a smooth function, 
taking  the values $\ve_{\rm p} $ and $\ve_{\rm w} $ in the solute 
region $\{ \phi \approx 1 \}$ and solvent region $\{ \phi \approx 0 \}$, respectively
\cite{LiLiu_SIAP15,Sun_PFVISM_JCP15}. 
The first variation of the functional $F_\xi[\phi]$ is given by \cite{LiLiu_SIAP15, Sun_PFVISM_JCP15} 
\begin{align}
\label{deltaphi}
\delta_\phi F_\xi [\phi] &= 2  P_0\, \phi +  \gamma_0
\left[-\xi \Delta\phi+\dfrac{1}{\xi } W'(\phi)\right] +  2\rho_0 (\phi-1)U
\nonumber \\
&\quad - \frac{1}{2} \varepsilon'(\phi) |\nabla\psi_\phi|^2 - 2 (\phi-1) B(\psi_\phi). 
\end{align}

We remark that the van der Waals--Cahn--Hilliard functional in the phase-field model
 \reff{Fxiphi} is exactly the interfacial free energy defined through the macroscopic
component of water density in the Lum-Chandler-Weeks solvation theory \cite{LCW99},
where though the electrostatics is not included.   
It has been recognized that such interfacial free energy is crucial in the description 
of hydrophobic interactions \cite{Chandler05,BerneWeeksZhou_Rev09,LCW99}. 

\subsection{Main Results and Connections to Existing Studies}
\label{ss:MainResults}

In this work, we study the limit properties of the phase-field free-energy functionals
\reff{Fxiphi} in terms of their sharp-interface limit. 
We prove the following: 
\begin{compactenum}
\item[(1)]
The convergence of the phase-field Poisson--Boltzmann electrostatics 
to the corresponding sharp-interface limit. More precisely, if a sequence of phase fields 
converge to a characteristic function of a subset of $\Omega$, 
then the corresponding sequences of electrostatic potentials, 
electrostatic free energies, and forces
converge to their respective sharp-interface counterparts; 
cf.~Theorem~\ref{t:PBenergy} and Theorem~\ref{th:f_ele-conv};   

\item[(2)]
The free-energy convergence. There are two main results concerning such convergence. 
First, the $\Gamma$-convergence of phase-field free-energy functionals to the corresponding
sharp-interface limit;  
cf.\ Theorem~\ref{t:EnergyConvergence}. 
The existence of a global minimizer of the sharp-interface free-energy functional $F$ 
is then a consequence of this $\Gamma$-convergence; cf.\ Corollary~\ref{c:existenceF0}.   
The proof of $\Gamma$-convergence is similar to that for the van der Waals--Cahn--Hilliard
functional. 
Care needs to be taken for the solute-solvent interaction part, 
i.e., the third term in \reff{FGamma} and that in \reff{Fxiphi}. 
In particular, we construct the recovering sequence as the same canonical phase fields
for the van der Waals--Cahn--Hilliard functional   
\cite{Modica_ARMA87,Sternberg_ARMA88}. 
Second, the equivalence of the convergence of total free energies and that 
of the individual parts of free energy (volume, surface, solute-solvent van der Waals interaction, 
and electrostatics); cf.\ Theorem~\ref{t:individual};  


\item[(3)]
The force convergence: if a sequence of phase fields converge to a characteristic function
and the corresponding solvation free energies converge to the sharp-interface free energy, 
then the corresponding phase-field forces converge to their sharp-interface counterpart. 
In fact, each individual part of the force converges to the corresponding sharp-interface part; 
cf.~Theorem~\ref{t:ForceConvSolvation}. 
There are two non-trivial parts in the proof of this force convergence. One is the proof
of electrostatic force convergence, which is Theorem~\ref{th:f_ele-conv}. 
The other is the proof of surface force convergence, i.e., the force convergence for 
the van der Waals--Cahn--Hilliard functional. Due to its general interest, we state and prove
a separate theorem,  Theorem~\ref{th:CH-force-conv},  for the surface force convergence. 
All the different kinds of forces are defined as the 
first variations of the corresponding parts of the free-energy functionals. 
These forces are shown to arise from stress tensors. Our results on force convergence are then 
stated in terms of the weak convergence of corresponding stress tensors.  

\end{compactenum}


Our work is closely related to the analysis in \cite{LiZhao_SIAP13} and \cite{LiLiu_SIAP15}.  
In \cite{LiZhao_SIAP13}, Li and Zhao study a similar but simpler phase-field 
model in which the electrostatic free energy is described by the Coulomb-field approximation 
 \cite{ChengChengLi_Nonlinearity11,Wang_VISMCFA_JCTC12},  
without the need of solving a dielectric Poisson or Poisson--Boltzmann equation. 
They obtain the $\Gamma$-convergence of the phase-field free-energy functionals to 
the respective sharp-interface functional. They also prove the existence of a global minimizer
of the sharp-interface free-energy functional. 
In \cite{LiLiu_SIAP15}, the authors obtain  the well-posedness of the
phase-field Poisson--Boltzmann equation and derive the variation \reff{deltaphi}. 
Using the matched asymptotic analysis, they also show 
that, in the sharp-interface limit as $\xi \to 0$,  the relaxation dynamics 
$ \phi_t = - \delta_\phi F_\xi [\phi]$ approaches that of the sharp-interface
governed by $v_n = -\delta_\Gamma F[\Gamma]$, where $v_n$ is the normal velocity of the sharp
boundary.  We shall use some of the results on the Poisson--Boltzmann 
electrostatics obtained in \cite{LiLiu_SIAP15}. 


We remark that the force convergence for (a subsequence of) 
van der Waals--Cahn--Hilliard functionals is proved in \cite{RogerSchatzle06}
under the assumption that corresponding sequence of free energy is bounded and that 
\begin{equation}
\label{phasefieldH2}
\sup_{0 < \xi \ll 1}  \int_\Omega \frac{1}{\xi}
\left[ - \xi \Delta \phi_\xi  +  \frac{1}{\xi} W'(\phi_\xi) \right]^2 dx < \infty,  
\end{equation}
where $\phi_\xi $ $(0 < \xi \ll 1)$ is the underlying family of phase fields; 
cf.\ also \cite{Sato_IndianaUnivJ08, Ilmanen1993, PadillaTonegawa_CPAM98,
HutchinsonTonegawa_2000, MizunoTonegawa_SIAM15, RogerSchatzle06} and the references therein. 
These assumptions provide additional regularities that 
allow one to show the equi-partition of the free energy, 
the existence of variation of the varifold corresponding to the limit of Radon measures 
\[
\left[ \frac{\xi}{2} |\nabla \phi_\xi |^2 + \frac{1}{\xi}  W(\phi_\xi) \right] dx,  
\]
and the rectifiability of the varifold. 
Here, we only assume the convergence of phase fields to a characteristic function and 
the corresponding convergence of the van der Waals--Cahn--Hilliard free energies 
to that of the sharp-interface counterpart, i.e., the perimeter of the limit set. 
The free-energy convergence is a natural assumption as the
free energies can converge to a different number even if the sequence of phase fields converge
to the same limit characteristic function; see an example constructed in 
Subsection~\ref{ss:Force}.  Our proof of force convergence involves no varifolds. 
It is rather based on the observation that
the free-energy convergence implies the asymptotic equi-partition of energy, and that 
the gradients of phase fields are controlled asymptotically by their projections onto 
the direction normal to the limit interface. 
Note that, without the additional assumption \qref{phasefieldH2}, 
we do not have the necessary regularities, and 
in turn we have to define the limit force in a weak sense through stress tensors. 
Consequently, the force convergence is proved as the weak convergence of stress tensors. 
\subsection{Organization of the Rest of Paper}
\label{ss:Organization}

In Section~\ref{s:MainResults}, we state our assumptions and main theorems. 
We also define forces and their corresponding stresses. 
In Section~\ref{s:PB}, we present results on the Poisson--Boltzmann electrostatics. 
These include a unified result on the well-posedness of the Poisson--Boltzmann equation, 
the continuity of the electrostatic free energy with respect to 
the change of dielectric regions, and the convergence of phase-field dielectric boundary
force to the sharp-interface limit.  In Section~\ref{s:FreeEnergyConvergence}, 
we prove the $\Gamma$-convergence of the phase-field free-energy functionals to their 
sharp-interface limit.  We also prove that the convergence of total free energies is equivalent
to that of individual parts of free energy.  Finally, in Section~\ref{s:ForceConvergenceSolvation}, 
we first prove the convergence of all the individual and
total phase-field forces to their sharp-interface counterparts  
for the solvation free-energy functional, except the surface force.  
We then focus on the proof of such surface 
that corresponds to the van der Waals--Cahn--Hilliard functional for 
a general $n$-dimensional space with $n \ge 2.$ 


\section{Main Theorems}
\label{s:MainResults}

\subsection{Assumptions}
\label{ss:Assumptions}


Unless otherwise stated, we assume the following throughout the rest of paper: 
\begin{compactenum}
\item[(A1)]
The set $\Omega \subset\R^3$ is nonempty, open, connected, and bounded with 
a $C^2$ boundary $\partial \Omega.$ 
The integer $N \ge 1$ and all points $x_1, \dots, x_N$ in $\Omega $  are given. 
All  $P_0$, $\gamma_0$, and $\rho_0  $ are positive numbers.  
The functions $\rho  \in H^1(\Omega)\cap L^\infty(\Omega)$ 
and $\psi_\infty \in W^{2,\infty}(\Omega)$ are given;  
\item[(A2)]
The function $U: \R^3 \to \R\cup \{ +\infty \}$ satisfies
\[
 U(x_i) = +\infty \quad \mbox{and} \quad 
 \lim_{x\to x_i} U(x) = +\infty \quad  ( i = 1, \dots, N), 
 \quad \mbox{and} \quad \lim_{x \to \infty} U(x) = 0.  
\]
Restricted onto $\R^3 \setminus \{ x_1, \dots, x_N \}$, $U$ is a $C^1$-function with 
\[
 U_{\rm min}: = \inf \{ U(x): x\in \R^3 \} \in (-\infty, 0].  
\]
Moreover, $U$ is not integrable in the neighborhood of each $x_i$ $(1\le i \le N)$
in the following sense:  for any measurable subset $\omega \subset \R^3$, 
\[
\int_{\omega}  U \, dx = +\infty  
\quad \mbox{if there exists } i\in\{ 1, \dots, N \} \mbox{ such that } 
 \inf_{r > 0} \frac{| \omega \cap B(x_i, r)|}{r^3} > 0,    
\]
where $|Q|$ denotes the Lebesgue measure of $Q$ in $\R^3;$
(In what follows, measure means the Lebesgue measure, unless otherwise stated.)
\item[(A3)]
The numbers $\ve_{\rm p}$ and $\ve_{\rm w}$ are positive and distinct.
The function $\ve \in C^1(\R)$ 
and it satisfies that $ \ve(\phi) = \ve_{\rm w}$ if $\phi \le 0$, $ \ve(\phi) = \ve_{\rm p}$ 
if $\phi \ge 1$, and $\ve(\phi) $ is monotonic in $(0, 1);$   
(Two examples of such a function $\ve$ are given in \cite{LiLiu_SIAP15}.)
\item[(A4)]
The function $B\in C^2(\R)$ is strictly convex with $B(0) = \min_{s \in \R} B(s) = 0.$
Moreover, $B(\pm \infty) = \infty$ and $B'(\pm \infty) = \pm \infty.$ 

\end{compactenum}

\subsection{Theorems on Free-Energy Convergence}
\label{ss:FreeEnergyConvergence}

We denote
\begin{equation}
\label{A}
\calA = \left\{ u \in H^1(\Omega):  u = \psi_\infty \mbox{ on } \partial \Omega \right\}.
\end{equation}
For any $\phi \in L^4(\Omega),$ we define $E_\phi: \calA \to \R \cup \{ \infty, -\infty \}$ by 
\begin{equation}
\label{Ephiu}
E_\phi [ u ] = \int_\Omega \left[ \frac{\ve (\phi) }{2} |\nabla u |^2
- \rho u  + (\phi-1)^2 B( u )  \right] dx.
\end{equation}
Since $B(u) \ge 0$, $E_\phi [u] > -\infty$ for any $u\in \calA.$
By Theorem~\ref{t:phiPB}, the functional $E_\phi: \calA \to \R \cup \{ +\infty \}$ 
has a unique minimizer $\psi_\phi \in \calA$ that is also the unique weak solution of the corresponding
boundary-value problem of the Poisson--Boltzmann equation:   
\reff{PBE} and \reff{BC} if $\phi$ is the characteristic function of the solute region with
boundary $\Gamma$; and \reff{PhaseFieldPBE} and \reff{PhaseFieldBC} if 
$\phi \in H^1(\Omega)$ is a general phase field. Moreover, in both cases, 
 \[
F_{\rm ele}[\phi] = -E_\phi[\psi_\phi] = -\min_{u\in \calA} E_\phi [u].  
\]
This is exactly the electrostatic free energy $F_{\rm ele}[\Gamma]$ defined in \reff{FeleGamma}
in the sharp-interface setting or $F_{\rm ele}[\phi]$ in \reff{Felephi} in the phase-field setting.  

Let us fix $\xi_0 \in (0, 1).$  We consider the phase-field functionals 
$ F_\xi:  L^1(\Omega)\to \R \cup \{\pm\infty\} $ for all $\xi \in (0, \xi_0]$ 
\cite{LiLiu_SIAP15,Sun_PFVISM_JCP15}:  
\begin{equation}
\label{newFxiphi}
F_\xi [\phi]   =
\left\{
\begin{aligned}
&  P_0 \int_\Omega \phi^2 \, d x 
+ \gamma_0 \int_\Omega  \left[ \frac{\xi }{2} |\nabla\phi|^2 + \frac{1}{\xi } 
W(\phi ) \right] dx + \rho_0   \int_\Omega ( \phi - 1 )^2  U \, dx + F_{\rm ele}[\phi] \\
& \qquad \qquad \, \mbox{if } \phi \in H^1(\Omega),  \\
&  +\infty \qquad  \mbox{otherwise}.  
\end{aligned}
\right.
\end{equation}
Note that $F_\xi$ never takes the value $-\infty$, as $U$ is bounded below and $F_{\rm ele}[\phi]$ 
is finite for any $\phi \in H^1(\Omega).$ 

Let $D$ be a nonempty, bounded, and open subset of $\R^n$ for some $n \ge 2.$ 
We recall that a function $u\in L^1(D) $ has bounded variations in $D$, if 
\begin{equation*}
|\nabla u|_{{BV}(\Omega)} 
:= \sup\left\{\int_{D}  u \, \mbox{div}\, g \, dx
: g\in C_{\rm c}^1(D, \R^n), |g|\le 1 \ \mbox{in } D  \right\} < \infty,
\end{equation*}
where $C_{\rm c}^1(D, \R^n)$ denotes the space of all
$C^1$-mappings from $D$ to $\R^n$ that are compactly supported inside $D$;
cf.\ \cite{Giusti84,Ziemer_Book89,EvansGariepy_Book92}.
If $u \in W^{1, 1}(D)$ then 
$ |\nabla u|_{BV(\Omega)} = \| \nabla u \|_{L^1(D)}$.  The space $BV(D)$ of all $L^1(D)$-functions that
have bounded variations in $D$ is a Banach space with the norm
\begin{align*}
\|u\|_{BV(D)}: = \|u\|_{L^1( D) } + |\nabla u|_{{BV}(D)} \qquad \forall u \in BV(D). 
\end{align*}
For any Lebesgue-measurable subset $A \subseteq \R^n$, the perimeter of $A$ in $D$ is defined by
\cite{Giusti84,Ziemer_Book89,EvansGariepy_Book92}
\[
 P_{D}(A):= |\nabla \chi_A|_{{BV}(D)}. 
\]

We define the sharp-interface free-energy functional 
$F_0: L^1(\Omega) \to \R \cup \{ \infty, -\infty\}$ by 
\begin{equation}
\label{def-F}
F_0[\phi] = \left\{
\begin{aligned}
& 
\displaystyle{
 P_0 |A|+\gamma_0 P_\Omega(A) + \rho_0 \int_{\Omega\setminus A} U \, dx 
+ F_{\rm ele}[\phi]
}
& & 
\quad  \mbox{if } \phi=\chi_A\in BV(\Omega),     &
\\
&  +\infty   & &     \quad \mbox{otherwise}.     & 
\end{aligned}
\right.
\end{equation}
If $\phi = \chi_A \in BV(\Omega),$ where $A\subset\Omega$ is an open subset 
with a smooth boundary $\Gamma$ and the closure $\overline{A} \subset \Omega$, 
then $F_0[\phi] = F[\Gamma]$ as defined in \reff{FGamma}. 
Note that the functional $F_0$ never takes the value $-\infty.$ 

We use the notation $\xi_k \searrow 0$ to indicate that $\{ \xi_k \}$
is a sequence of real numbers such that $\xi_1 > \xi_2 > \cdots$ and
$\xi_k \to 0$ as $k \to \infty.$ We always assume that $\xi_1 \in (0,\xi_0]$.
The following theorem on free-energy convergence is proved in 
Section~\ref{s:FreeEnergyConvergence}:  

\begin{theorem}[$\Gamma$-convergence of free-energy functionals]
\label{t:EnergyConvergence}
For any sequence 
$\xi_k \searrow 0$, the sequence of functionals $F_{\xi_k}: L^1(\Omega)\to \R \cup \{ +\infty \}$ 
$(k = 1, 2, \dots)$
$\Gamma$-converges to the functional $F_0: L^1(\Omega) \to \R \cup \{ + \infty \}$ 
with respect to the $L^1(\Omega)$-convergence. 
This means precisely that the following two properties hold true: 
\begin{compactenum}
\item[\rm (1)]
{\rm The liminf condition.} 
If $\phi_k \to \phi$ in $L^1(\Omega)$ then 
\begin{equation}
\label{liminf-ineq}
\liminf_{k \to \infty} F_{\xi_k}[\phi_k] \ge F_0[\phi]; 
\end{equation}
\item[\rm (2)]
{\rm The recovering sequence.} 
For any $\phi \in L^1(\phi)$, there exist $\phi_k \in L^1(\Omega)$ $(k = 1, 2, \dots)$ such 
that $\phi_k \to \phi$ in $L^1(\Omega)$ and 
\begin{equation}
\label{limsup-ineq}
\limsup_{k \to \infty} F_{\xi_k}[\phi_k] \le F_0[\phi]. 
\end{equation}
\end{compactenum}
\end{theorem}


We remark that this result does not follow immediately from the stability of $\Gamma$-convergence
under continuous perturbations. 
In fact, the solute-solvent interaction term (i.e., the third term) and the electrostatics
term (i.e., the fourth term) in the phase-field functional \reff{newFxiphi}
are not simple continuous perturbations of the van der Waals--Cahn--Hilliard functionals. 
The convergence of those terms require more than the $L^1(\Omega)$-convergence
of underlying phase-field functions. 

The following corollary of the above theorem provides the existence of 
minimizers of the corresponding sharp-interface free-energy functional: 



\begin{corollary}
\label{c:existenceF0}
There exists a measurable subset $G \subseteq \Omega$ with finite perimeter $P_\Omega(G)$ in $\Omega$
such that $F_0[\chi_G] = \min_{\phi \in L^1(\Omega)} F_0[\phi],$ which is finite.  
\end{corollary}

The next result, also proved in Section~\ref{s:FreeEnergyConvergence}, is of interest by itself.
It states that each component of the
free energy converges to its sharp-interface analog, if the total free energy converges.

 
\begin{theorem}
\label{t:individual}
Let $\xi_k \searrow 0,$ $\phi_k\in H^1(\Omega)$ $(k = 1, 2, \dots),$ 
and $G\subseteq \Omega$ be measurable with $P_\Omega(G) < \infty.$ 
Assume that $\phi_k \to \chi_G $ a.e.\ in $ \Omega $ 
and $F_{\xi_k}[\phi_k] \to F_0 [\chi_G] $ with $F_0[\chi_G]$ finite. Then
\begin{align}
\label{volume}
&\lim_{k\to \infty} \int_\Omega \phi_k^2 dx = |G|, 
\\
\label{surface}
&\lim_{k\to \infty} \int_\Omega \left[ \frac{\xi_k}{2} |\nabla \phi_k|^2 
+ \frac{1}{\xi_k} W(\phi_k) \right] dx = P_\Omega(G),  
\\
\label{LJ}
&\lim_{k\to \infty} \int_\Omega (\phi_k-1)^2 U \, dx = \int_{\Omega\setminus G} U\, dx,  
\\
\label{Ele}
&\lim_{k\to \infty} F_{\rm ele}[\phi_k] = F_{\rm ele}[\chi_G].
\end{align}
All the limits are finite. 
\end{theorem}

\subsection{Definition of Force and Theorems on Force Convergence}
\label{ss:Force}


\subsubsection{Force in the Phase-Field Model}

Let $\xi \in (0, \xi_0]$. 
We define the individual forces as vector-valued functions on $\Omega$ as follows:
\begin{align*}
& f_{\rm vol} (\phi) = 2 P_0 \phi \nabla \phi 
& & \mbox{if } \phi \in H^1(\Omega), 
\\
& f_{\xi, {\rm sur}} (\phi) = \gamma_0\left[ -\xi \Delta \phi  
+ \frac{1}{\xi } W'(\phi) \right] \nabla \phi & & \mbox{if } \phi \in H^2(\Omega), 
\\
& f_{\rm vdW}(\phi) = 2 \rho_0 (\phi - 1) U \nabla \phi
& & \mbox{if } \phi \in H^1(\Omega), 
\\
&  f_{\rm ele}(\phi) = \left[ - \frac{\ve'(\phi)}{2}|\nabla \psi_\phi|^2  
  - 2 (\phi - 1)B(\psi_\phi) \right] \nabla \phi  
& & \mbox{if } \phi \in H^1(\Omega), 
\end{align*}
where $\psi_\phi \in \calA$ is electrostatic potential corresponding to $\phi$, 
i.e., the solution to the boundary-value problem of Poisson--Boltzmann equation
\reff{PhaseFieldPBE} and \reff{PhaseFieldBC}; cf.~Theorem~\ref{t:phiPB}. 
If $\phi \in H^2(\Omega)$, we define the total force 
\begin{equation}
\label{f_xi_total}
f_\xi(\phi) = 
f_{\rm vol} (\phi) + f_{\xi, {\rm sur}} (\phi) + f_{\rm vdW}(\phi) + f_{\rm ele}(\phi). 
\end{equation}
Note that these forces are given as $-\nabla \phi$ multiplied by the negative first variations of the
volume, surface, van der Waals solute-solvent interaction, electrostatics, and the total
free energy, respectively; cf.~\reff{deltaphi}. 
Note also that a phase field $\phi$ of lower free energy is close to the 
characteristic function of solute region. The direction $-\nabla \phi$ then points from 
the solute to solvent region, same as the direction $\nu$ in the sharp-interface 
force \reff{BoundaryForce}.  



The forces can be also defined by the method of domain variations. 
Given $V\in C_c^1(\Omega, \R^n)$, we define $x = x(t,X)$ with $t \in (-t_0, t_0)$ for 
some $t_0 > 0$ small and $X \in \Omega$ by
$\dot{x} = V(x)$ and $ x(0,X) = X. $  
This defines a family of transformations $T_t: \Omega \to \Omega$ with $T_t(X) =  x (t,X). $
For a smooth phase field $\phi$, these transformations define the perturbations
$\phi\circ T_t$ of $\phi.$  For the phase-field functional $F_\xi $, one then defines naturally 
the force to be $-(d/dt)|_{t=0} F_\xi [\phi \circ T_t]$, 
the negative variation of the phase-field free-energy functional $F_\xi$ at $\phi$
with respect to these perturbations. 
Note that 
\[
T_t(X) = X + t V(X) + o(t)  \quad \mbox{as } t \to 0.
\]
Hence, 
\[
(\phi \circ T_t)(X)  = \phi (X) +  t \nabla \phi (X) \cdot V(X) + o(t) \qquad\mbox{as } t \to 0.
\]
Therefore,  
\[
- \frac{d}{dt}\biggr|_{t=0} F_\xi [\phi \circ T_t] = 
- \frac{d}{dt}\biggr|_{t=0} F_\xi [\phi + t \nabla \phi \cdot V + o(t)]
=  - \delta_\phi F_\xi [\phi] \nabla \phi \cdot V. 
\]
By \reff{f_xi_total}, this differs from $- f_\xi (\phi) \cdot V$ 
only by a sign. 
This sign difference results from our choice of force direction as discussed above.

We now define the corresponding individual 
stress tensors (with respect to the underlying coordinate system) by 
\begin{align}
\label{stressphivol}
& T_{\rm vol}(\phi) = P_0 \phi^2 I && \mbox{if }\phi \in L^4(\Omega), \\
\label{stressphisurf}
&T_{\xi, {\rm sur} } (\phi) =  \gamma_0\left\{ \left[ 
 	\frac{\xi}{2} |\nabla \phi|^2 + \frac{1}{\xi} W(\phi)\right] I- \xi \nabla \phi \otimes \nabla \phi \right\}
	&& \mbox{if }\phi \in H^1(\Omega),\\
\label{stressphivdW}
&T_{\rm vdW}(\phi) = \rho_0 (\phi-1)^2 U I
 && \mbox{if }\phi \in L^4(\Omega), 
\\
\label{stressphielec}
&T_{\rm ele}(\phi) = \ve(\phi) \nabla \psi_\phi \otimes \nabla \psi_\phi 
- \left[   \frac{\ve(\phi)}{2} | \nabla \psi_\phi |^2  + (\phi - 1)^2 B(\psi_\phi) \right] I
&& \mbox{if }\phi \in L^4(\Omega).  
\end{align}
Note that we assume $\phi \in L^4(\Omega)$, as our 
double-well potential $W = W(\phi)$ defined in \reff{W} is a polynomial of degree $4$. 
Moreover, that $\phi\in L^4(\Omega)$ is necessary for 
the term $(\phi-1)^2$ in the functional $F_{\xi}[\phi] $ 
defined in \reff{Fxiphi} and $F_{\rm ele}[\phi] $ defined in \reff{Felephi} to be in $L^2(\Omega)$. 
Note also that we have the Sobolev embedding $H^1(\Omega) \hookrightarrow L^4(\Omega).$ 

We recall that the divergence of a tensor field $T = (T_{ij})$, denoted $\nabla \cdot T$ 
or $\mbox{div}\, T$, is the vector field  with components $\partial_j T_{ij}$ 
$(i = 1, 2, 3)$, if exist.  For a differentiable vector field $V: \Omega \to \R^3$ 
that has components $V_i $ $(i = 1, 2, 3)$, the gradient $\nabla V$ 
is the matrix-valued function with the $(i,j)$-entry $\partial_j V_i.$ 
For any $3 \times 3$ matrices $A$ and $B,$ we define  $A : B = \sum_{i,j=1}^3 A_{ij} B_{ij}.$ 
We also define $|A|$ by $|A|^2 = {\sum_{i,j=1}^3 |A_{ij}|^2}.$ 
It is straightforward to generalize these definition and notation to $\R^n$ 
for any $n \ge 2.$ 

The following lemma indicates that the phase-field forces defined above arise from the corresponding
stress tensors. Moreover, lower regularities of phase field $\phi$ 
are needed to define the stress tensors: 



\begin{lemma}
\label{l:Stress}
We have for almost all points in $\Omega$ that 
\begin{align}
\label{Tfvol}
&f_{\rm vol}(\phi) = \nabla\cdot T_{\rm vol}(\phi) 
& & \mbox{if } \phi \in H^1(\Omega), &
\\
\label{Tfsur}
&f_{\xi, {\rm sur}}(\phi) =\nabla \cdot T_{\xi, {\rm sur}}(\phi) 
&&  \mbox{if } \phi \in H^2(\Omega), &
\\
\label{TfvdW}
&f_{\rm vdW}(\phi) = \nabla \cdot T_{\rm vdW }(\phi) - \rho_0 (\phi-1)^2 \nabla U 
&&  \mbox{if } \phi \in H^1(\Omega), &
\\
\label{Tfele}
&f_{\rm ele}(\phi)=\nabla \cdot T_{\rm ele}(\phi) +\rho \nabla \psi_\phi
&&  \mbox{if } \phi \in W^{1,\infty}(\Omega).  &
\end{align}
Moreover, we have for any $V \in C_c^1(\Omega,\R^3)$ that 
\begin{align}
\label{weak-f_vol}
&  \int_\Omega f_{\rm vol}(\phi) \cdot  V \, dx
=   - \int_\Omega T_{\rm vol} (\phi): \nabla V\, dx 
\qquad  \mbox{if } \phi \in H^1(\Omega), 
\\
\label{weak-f_sur}
&  \int_\Omega  f_{\xi, {\rm sur}}  (\phi) \cdot V \, dx 
=   - \int_\Omega T_{\xi, {\rm sur}}(\phi) : \nabla V\, dx 
\qquad  \mbox{if } \phi \in H^2(\Omega), 
\\
\label{weak-f_vdW}
&  \int_\Omega f_{\rm vdW } (\phi) \cdot  V \, dx 
=  - \int_\Omega \left[ T_{\rm vdW} (\phi): \nabla V
	+ \rho_0  (\phi-1)^2 \nabla U \cdot V\right]   dx
\nonumber \\
& \qquad \qquad  \qquad \qquad \qquad  
 \mbox{if }  \{ x_1, \dots, x_N \} \cap \mbox{\rm supp}\, (V) =  \emptyset 
\quad \mbox{and} \quad \phi \in H^1(\Omega), 
\\
\label{weak-f_ele}
& \int_\Omega f_{\rm ele}(\phi) \cdot  V \, dx 
=  - \int_\Omega \left[ T_{\rm ele}(\phi) :\nabla V  
	-\rho \nabla \psi_\phi \cdot V\right]  \, dx  
\qquad   \mbox{if } \phi \in W^{1, \infty} (\Omega).  
\end{align}
\end{lemma}

\begin{proof}
The identities \reff{Tfvol} and \reff{TfvdW} follow from direct calculations. 
All the identities \reff{weak-f_vol}--\reff{weak-f_ele} follow from \reff{Tfvol}--\reff{Tfele} 
and integration by parts.  Therefore, it remains only prove \reff{Tfsur} and \reff{Tfele}. 

Let $\phi\in H^2(\Omega)$ and $i \in \{ 1, 2, 3\}.$  We have 
by the definition of $T_{\xi, {\rm sur}}(\phi)$ and using the summation convention that 
\begin{align*}
\partial_j T_{\xi, {\rm sur}, ij}(\phi)
&= \gamma_0 \partial_j \left\{ 
\left[ \frac{\xi}{2} \partial_k \phi \partial_k \phi + \frac{1}{\xi } W(\phi)\right] \delta_{ij}
 - \xi \partial_i \phi \partial_j \phi
\right\} \\
& =\gamma_0\left\{ \xi \partial_{ik} \phi \partial_k \phi + \frac{1}{\xi} W'(\phi) \partial_i \phi 
-\xi \partial_{ij}\phi \partial_j \phi - \xi \partial_i \phi \Delta \phi \right\}\\
& =\gamma_0 \left[ - \xi \Delta \phi + \frac{1}{\xi} W'(\phi) \right] \partial_i \phi,
\end{align*}
where $\delta_{ij}  = 1 $ if $i = j$ and $0$ otherwise. 
This is the $i$th component of the force vector $f_{\xi, {\rm sur}}$; \reff{Tfsur} is thus proved. 

Now let $\phi\in W^{1,\infty}(\Omega).$ By Theorem~\ref{t:phiPB},
$\psi_\phi$ is bounded on $\chi_{\{ \phi \ne 1\}}$. Since $\phi \in W^{1,\infty}(\Omega),$ we have  
\[
\ve(\phi) \Delta \psi_\phi
 = -\rho - \ve'(\phi) \nabla  \phi \cdot \nabla \psi_\phi 
+ (\phi-1)^2 B'(\psi_\phi) \in L^2(\Omega).  
\]
Hence $\psi_\phi \in H^2(\Omega).$ 
By direct calculations using the fact that $\psi_\phi$ solves the Poisson--Boltzmann 
equation, we obtain 
\begin{align*}
 \partial_j T_{{\rm ele}, ij} &
= \partial_j ( \ve(\phi) \partial_i \psi_\phi \partial_j \psi_\phi ) - 
\delta_{ij} \partial_j \left[ \frac12 \ve(\phi) \partial_k \psi_\phi \partial_k \psi_\phi
+ (\phi-1)^2 B(\psi_\phi) \right] \\
	& = \ve'(\phi) \partial_j \phi \partial_i \psi_\phi \partial_j \psi_\phi
		+ \ve(\phi) \partial_{ij} \psi_\phi \partial_j \psi_\phi 
		+\ve(\phi) \partial_i \psi_\phi \Delta \psi_\phi \\
	&\qquad - \frac12 \ve'(\phi) \partial_i \phi |\partial \psi_\phi|^2 -
\ve(\phi) \partial_{ik} \psi_\phi \partial_k \psi_\phi  - 2 (\phi-1) \partial_i \phi B(\psi_\phi) 
- (\phi-1)^2 B'(\psi_\phi) \partial_i \psi_\phi \\
& = \left[  \nabla \cdot \ve(\phi) \nabla \psi_\phi 
- (\phi-1)^2 B'(\psi_\phi) \right] \partial_i \psi_\phi -
\left[ \frac{\ve'(\phi)}{2} | \nabla \psi_\phi |^2 + 
2 (\phi - 1 ) B(\psi_\phi) \right] \partial_i \phi
\\
&=- \rho \partial_i \psi_\phi - \left[ \frac{\ve'(\phi)}{2} | \nabla \psi_\phi |^2 
+ 2 (\phi - 1 ) B(\psi_\phi) \right] \partial_i \phi, \qquad  i = 1, 2, 3,
\end{align*}
proving \reff{Tfele}. 
\end{proof}

\subsubsection{Force in the Sharp-Interface Model}


Let $G$ be an open subset of $\Omega$ such that the closure $\overline{G} \subset \Omega$, 
the boundary $\partial G$ is $C^2,$ and $x_i \in G$ $(i = 1, \dots, N).$ 
Denote by $\nu$ the unit vector on $\partial G$ that points from $G$ to $G^c = \Omega \setminus G.$  
Following \reff{FGamma} (with $\Gamma = \partial G$) or \reff{def-F} (with $A = G$), 
 and \reff{BoundaryForce} (with $\Gamma = \partial G$),  
we define the individual volume, surface, van der Waals, and electrostatic forces on 
the boundary $\partial G$ as vector-valued functions on $\partial G$ as follows:
\begin{align}
	& f_{0,{\rm vol}}[\partial G] = -P_0 \nu,  \label{f0_vol}\\
	& f_{0, {\rm sur} } [\partial G]=  - 2\gamma_0 H\nu, \label{f0_sur} \\
	& f_{0,{\rm vdW}}[\partial G] =\rho_0U\nu, \label{f0_vdW} \\
& f_{0, {\rm ele}}[\partial G] 
= \left[ - \frac{1} {2  } \left( \frac{1}{\ve_{\rm p}} - \frac{1}{\ve_{\rm w}} \right)
 \left|\ve (\chi_G) \nabla  \psi_{\chi_G}  \cdot \nu  \right|^2 \right.
\nonumber \\
&\qquad \qquad \quad \left.  - \frac12 (\ve_{\rm w} - \ve_{\rm p}) \left| (I - \nu \otimes \nu ) 
\nabla \psi_{\chi_G} \right|^2 - B(\psi_{\chi_G}) \right] \nu.   \label{f0_ele}
\end{align}
We also define the total boundary force to be 
\begin{align*}
	f_0[\partial G]&= f_{0,{\rm vol}}[\partial G] + f_{0,{\rm sur}} [\partial G]+f_{0,{\rm vdW}}[\partial G] 
	+f_{0, {\rm ele}}[\partial G].
\end{align*}
In \qref{f0_sur}, $H$ is the mean curvature of $\partial G$, 
defined as the average of the principal curvatures, and is positive if $G$ is convex. 
In \reff{f0_ele}, $\psi_{\chi_G}\in \calA$ is the electrostatic potential
corresponding to $\chi_G$; cf.~Theorem~\ref{t:phiPB}.  
It satisfies $\psi_{\chi_G} |_G \in H^2(G)$ and $\psi_{\chi_G} |_{G^c} \in H^2(G^c)$. 
Moreover (cf.~\cite{Li_SIMA09,LiChengZhang_SIAP11}),    
\begin{align}
\label{psiGp}
& - \ve_{\rm p} \Delta \psi_{\chi_G} = \rho & &  \mbox{in } G, & \\
\label{psiGw}
& - \ve_{\rm w} \Delta \psi_{\chi_G}  + B'(\psi) =  \rho  &&  \mbox{in } G^c, &  \\
\label{psiGcont}
& \psi_{\chi_G} |_{G} = \psi_{\chi_G} |_{G^c} & &  \mbox{on } \partial G, & \\
\label{psiGgrad}
& \ve_{\rm p} \nabla \psi_{\chi_G} |_{G}\cdot \nu  =  
\ve_{\rm w} \nabla \psi_{\chi_G} |_{G^c} \cdot \nu
&& \mbox{on } \partial G.       & 
\end{align} 
The quantity $\ve(\chi_G) \nabla \psi_{\chi_G} \cdot \nu $ in \reff{f0_ele} is 
the common value of both sides of \reff{psiGgrad}.  
By \reff{psiGcont}, the tangential gradient $(I-\nu \otimes \nu)\nabla \psi_{\chi_G} $ 
in \reff{f0_ele} is the same when $\psi_{\chi_G}$ is restricted onto either 
side of the boundary $\partial G.$ 


We recall that the stress tensors $T_{\rm vol}(\chi_G)$, $T_{\rm vdW}(\chi_G)$, 
and $T_{\rm ele}(\chi_G)$ are defined in \reff{stressphivol}, \reff{stressphivdW}, and 
\reff{stressphielec}, respectively, with $\phi $ replaced by $\chi_G.$ 
The following lemma indicates that the forces defined above in 
\reff{f0_vol}--\reff{f0_ele} also arise from stress tensors in the sharp-interface model and that
only lower regularity of the subset $G$ is needed to define the stresses:  


\begin{lemma}  
\label{l:sharpboundaryforce}
Let $G$ be an open subset of $\Omega$ such that the closure $\overline{G} \subset \Omega$ 
and the boundary $\partial G$ is $C^2.$ 
Let $\nu$ denote the unit vector $\nu$ on $\partial G$ that points from $G$ to $G^c.$  
We have for any $V\in C_c^1(\Omega,\R^n)$ that 
\begin{align}
& \int_{\partial G} f_{0, {\rm vol}}[\partial G] \cdot V \, dS
	= -\int_\Omega T_{\rm vol}(\chi_G):\nabla V \,dx, 
\label{weak-f0_vol}\\
& \int_{\partial G} f_{0, {\rm sur}}[\partial G] \cdot V \, dS
= -\gamma_0\int_{\partial G} (I-\nu\otimes\nu):\nabla V\,d S,  \label{weak-f0_sur} \\
& \int_{\partial G} f_{0, {\rm vdW}}[\partial G] \cdot V \, dS
= -\int_{\Omega} \left[  T_{\rm vdW}(\chi_G):\nabla V +\rho_0(1-\chi_G)^2\nabla U\cdot V\right]  \,dx  
\nonumber \\
& \qquad \qquad \qquad \qquad \qquad \qquad \qquad 
 \mbox{if }  \{ x_1, \dots, x_N \} \cap \mbox{\rm supp}\, (V) =  \emptyset, 
		\label{weak-f0_vdW}\\
& \int_{\partial G} f_{0, {\rm ele}}[\partial G] \cdot V \, dS
=- \int_\Omega \left[ T_{\rm ele}(\chi_G) : \nabla V -\rho \nabla \psi_{\chi_G} \cdot V \right] \, dx.  
		\label{weak-f0_ele}
\end{align}
\end{lemma}

\begin{proof}
Eq.~\reff{weak-f0_vol} follows from the identity $I:\nabla V = \nabla \cdot V$ 
and an application of the divergence theorem.  
Eq.~\reff{weak-f0_sur} follows from our definition of force $f_{0, {\rm sur}}$
and the known result (cf.\ Lemma~10.8 in \cite{Giusti84}): 
\[
 \int_{\partial G} 2 H \nu \cdot V \, dS = \int_{\partial G} (I - \nu \otimes \nu) : \nabla V \, dS.  
\]  
Assume each $x_i \not\in \mbox{supp}\,(V) $ $(1 \le i \le N).$ Noticing that
$\nu$ points from $G$ to $G^c=\Omega\setminus G$, we have by the definition of $T_{\rm vdW}(\chi_G)$ 
(cf.\ \reff{stressphivdW}) and the divergence theorem that 
\begin{align*}
&\int_{\Omega} \left[  T_{\rm vdW}(\chi_G):\nabla V +\rho_0(1-\chi_G)^2\nabla U\cdot V\right] \,dx  
\\
&\qquad 
 = \rho_0 \int_{G^c} ( U \nabla \cdot V + \nabla U \cdot V) \, dx
\\
&\qquad 
= \rho_0 \int_{G^c} \nabla ( U V) \, dx 
\\
&\qquad 
= -\rho_0 \int_{\partial G} U \nu \cdot V\, dS, 
\end{align*}
leading to \reff{weak-f0_vdW}. 
Finally, Eq.~\reff{weak-f0_ele} is part of Theorem~\ref{th:f_ele-conv} that is 
proved in Section~\ref{s:PB}. 
\end{proof}

\subsubsection{Force Convergence}


Let $D$ be a nonempty, open, and bounded subset of $\R^n$ with $n \ge 2.$
For any measurable subset $G$ of $D $ with $\overline{G} \subset D $ and 
$P_D (G) < \infty,$  we denote by $\partial^*G $ the reduced boundary of $G$ and  
by $\| \partial G\|=\cH^{n-1}\mres(\partial^*G\cap D )$ the perimeter measure of $G$ 
in $D$, where $\cH^{n-1}$ denotes the $(n-1)$-dimensional Hausdorff measure  
\cite{Giusti84,EvansGariepy_Book92,Ziemer_Book89}.  
We also denote  by $\nu: D \to \R^n$  the unit outer normal of $\partial^*G$.  
We recall that $|\nu| = 1$ $\|\partial G\|$-a.e.\ and  
\begin{equation}
\label{perimetermeasure}
\int_G \nabla \cdot g \, dx = \int_{\partial^*G} g \cdot  \nu \, d\calH^{n-1}
\qquad \forall g \in C_c^1(\Omega, \R^n). 
\end{equation}

The following result states that the convergence of total force is equivalent
to that of individual forces; 
its proof is given in Section~\ref{s:ForceConvergenceSolvation}: 


\begin{theorem}[Force convergence for the solvation free-energy functional]
\label{t:ForceConvSolvation}
Let $G$ be  a measurable subset of $\Omega$ such that $\overline{G} \subset\Omega$, 
$P_\Omega(G) < \infty,$ 
and $F_0[\chi_G]$ is finite.  
Let $\xi_k\searrow 0$ and $\phi_k \in H^1(\Omega) $ $(k = 1, 2, \dots)$ be such that  
$\phi_k \to \chi_G$ a.e.\ in $\Omega$ and $ F_{\xi_k} [\phi_k] \to  F_0[\chi_G].  $
Then we have for any $V \in C_c^1(\Omega, \R^3)$ that 
\begin{align}
& \lim_{k \to \infty} \int_\Omega T_{\rm vol}  (\phi_k) : \nabla V \, dx 
= \int_{\Omega} T_{\rm vol}(\chi_G):\nabla V \, dx, 
	\label{f_vol-conv}\\
& \lim_{k \to \infty} \int_\Omega T_{\xi_k, {\rm sur}} (\phi_k) : \nabla V \, dx 
= \gamma_0\int_{\partial^* G} \left( I-\nu\otimes\nu \right):\nabla V\, d\calH^{2}, 
\label{CH-force-conv}
\\
&\lim_{k \to \infty}  \int_\Omega \left[ ( T_{\rm vdW} (\phi_k) : \nabla V
 + \rho_0 ( \phi_k - 1)^2 \nabla U \cdot V \right] dx  \nn\\
&\quad= \int_\Omega \left[  T_{\rm vdW} (\chi_G) : \nabla V
 + \rho_0 ( \chi_G - 1)^2 \nabla U \cdot V \right]  dx  
 \quad  \mbox{if }  
\{ x_1, \dots, x_N \} \cap \mbox{\rm supp}\, (V) =  \emptyset,  
	\label{f_vdW-conv}\\
& \lim_{k\to \infty} \int_\Omega \left[  T_{\rm ele}(\phi_k) 
: \nabla V -\rho \nabla \psi_{\phi_k} \cdot V \right]  \, dx 
= \int_\Omega \left[  T_{\rm ele}(\chi_G) : \nabla V -\rho \nabla \psi_{\chi_G} \cdot V \right]
 \, dx.  \label{f_ele-conv}
\end{align} 
\end{theorem}

The force convergence for the van der Waals--Cahn--Hilliard functional is the main
part of the above theorem. Since this functional is rather a general model, we 
state separately the result of its force convergence for a general $n$-dimensional space. 
For simplicity of notation, we define the stress tensor 
$T_{\xi}(\phi)$ to be the same as $T_{\xi,{\rm sur}}(\phi)$ defined in \reff{stressphisurf}, except we 
take $\gamma_0 = 1$, i.e., we define for a function $\phi$ of $n$-variables
\begin{equation*}
T_{\xi} (\phi) =  \left[ \frac{\xi}{2} |\nabla \phi|^2 + \frac{1}{\xi} W(\phi)\right] I
- \xi \nabla \phi \otimes \nabla \phi, 
\end{equation*}
where $I$ is the $n\times n$ identity matrix. 


\begin{theorem}[Force convergence for the van der Walls--Cahn--Hilliard functional]
\label{th:CH-force-conv}
Let $\Omega $ be a nonempty, bounded, and open subset of $\R^n.$
Let $G$ be a nonempty,  measurable subset of $\Omega$ 
such that $\overline{G} \subset \Omega$ and $P_\Omega (G) < \infty$. 
Assume $\xi_k \searrow 0$ and 
 $\phi_k \in H^1(\Omega)$ $(k = 1, 2, \dots)$ satisfy that 
$\phi_k \to \chi_G$ a.e.\ in $\Omega$ and that 
\begin{equation}
\label{important}
\lim_{k \to \infty} \int_\Omega \left[ \frac{\xi_k}{2} |\nabla \phi_k|^2 
+ \frac{1}{\xi_k} W(\phi_k) \right] dx = P_\Omega(G). 
\end{equation}
Then we have for any $\Psi \in C_c(\Omega, \R^{n\times n})$ that
	  	\begin{align} \label{CH-force-conv1}
			\lim_{k \to \infty} \int_\Omega T_{\xi_k} (\phi_k) : \Psi \, dx 
			= \int_{\partial^* G} \left(I- \nu\otimes\nu\right):\Psi\, d\calH^{n-1}.   
		\end{align}
	If, in addition, $\phi_k\in W^{2,2}(\Omega)$ $(k = 1, 2, \dots)$, 
 $G$ is open,  and $\partial G$ is of $C^2$, then we have for any $V\in C_c^1(\Omega,\R^n)$ that 
		\begin{align}\label{CH-force-conv2}
			\lim_{k\to\infty}\int_\Omega\left[  -\xi_k\Delta\phi_k 
+ \frac1{\xi_k} W'(\phi_k)\right] \nabla\phi_k \cdot V\;dx
			= - (n-1)\int_{\partial G} H\nu\cdot V \, dS.
		\end{align}
\end{theorem}

We remark that the assumption of the above theorem requires 
the convergence of free-energy, i.e., \reff{important}. 
Such convergence is not guaranteed by the assumptions that $\phi_k \to \chi_G$ a.e.\ in $\Omega$ 
and $\phi_k \to \chi_G$ in $L^1(\Omega)$. This is expected as not every
such sequence is a recovery sequence of the $\Gamma$-convergence.
In particular, let $G$ be an open subset of $\Omega$ with a smooth boundary $\partial G$ and
$\overline{G} \subset \Omega$, and let $\beta $ be any real number such that 
\[
\beta\geq \sigma :=\int_0^1\sqrt{2W(s)}\;ds. 
\] 
(We have $\sigma = 1$ for our choice of $W$.)
We show that there exist $\phi_k\in H^1(\Omega)$ $(k = 1, 2, \dots)$ such that 
 \begin{enumerate}
 	\item[(1)] 
$\phi_k\to \chi_G$ a.e.\ in $\Omega$ and $\phi_k\to \chi_G$ in $L^1(\Omega)$;
	\item[(2)]		
$\disp \lim_{k\to\infty} \int_\Omega \left[ \frac{\xi_k}{2}|\nabla\phi_k|^2 
		+\frac{W(\phi_k)}{\xi_k} \right]  dx =\beta P_{\Omega}(G).$ 
\end{enumerate}

Let $a > 0$ and define $W_a (s) = {W(s)}/a$ $(s\in \R).$ 
For each $k \ge 1$, we define $q_k: [0,1]\to \R$ by
\[
q_k (t) = \int_0^t \frac{\xi_k }{\sqrt{2 [ W_a(\tau ) + \xi_k] }}\, d\tau  \qquad \forall t\in[0,1].
\]
Clearly, $q_k $ is a strictly increasing function of $t \in [0, 1]$  with $q_k (0)  = 0,$
$\lambda_k  := q_k (1)  \in (0, \sqrt{\xi_k /2}),$ and $q_k(t) \le t$ for any $t \in [0,1].$  
Let $g_k : [0,\lambda_k ]\rightarrow[0,1]$ be the inverse of $q_k: [0,1]\to [0,\lambda_k ]$.
By using the formula of derivatives of inverse functions, we obtain
\begin{equation*}
 g_k'( s ) = \frac{1}{\xi_k} \sqrt{ 2 [ W_a  (g_k(s)) + \xi_k] } \qquad \forall  s \in[0,\lambda_k ]. 
\end{equation*}
We extend $g_k $  onto the entire real line by defining
$g_k (s) = 0$ for any $ s < 0$ and $g_k (s) = 1$ for any $ s > \lambda_k.$
Denote now by $d: \Omega \to\R$ the signed distance function to the boundary $\partial G:$
$d(x) = \mbox{dist}\,(x, \partial G)$ if $ x \in G$ and 
$d(x) = -\mbox{dist}\,(x, \partial G)$ if $ x \in G^c$.  Let $\xi_k \searrow 0.$
Define $\phi_k: \Omega \to [0, 1]$ by  $\phi_k (x) = {g}_k (d(x))$ $(x\in \Omega)$. 
Then $\phi_k \to \chi_G$ a.e.\ in $\Omega$ and $\phi_k \to \chi_G$ in $L^1(\Omega)$ 
 \cite{Modica_ARMA87,Sternberg_ARMA88}.
Moreover, since $\partial G$ is smooth, we have for a.e.\ $x\in \Omega$ and $k$ large enough that 
\begin{align*}
|\nabla\phi_k(x)| = |  g_k'(d(x))\nabla d(x)| = 
\frac{1}{\xi_k}\sqrt{ 2 \left[  W_a (\phi_k(x)) +\xi_k \right] }.
\end{align*}

Note for any $s \in [0, 1]$ that $\phi_k(x) = s$ if and only if $ d(x) = q_k(s)$, 
and $q_k(s) \le \lambda_k \to 0$ as $k \to \infty.$ 
Since $\partial G$ is smooth, we have (cf.\ Lemma~4 in \cite{Modica_ARMA87}
and Lemma~2 in \cite{Sternberg_ARMA88}) that 
\begin{align*}
& 
\lim_{k\to \infty} \sup_{0 \le s \le 1 } \calH^{n-1}( \{ x\in \Omega: \phi_k (x) = s \} )
\\
&\qquad 
= \lim_{k\to \infty} \sup_{0 \le s \le 1} \calH^{n-1}( \{ x\in \Omega: d(x) = q_k(s) \} )
\\
&\qquad 
= P_\Omega(G).
\end{align*}
Consequently, applying the co-area formula and the Lebesgue Dominated Convergence Theorem, 
we obtain that 
\begin{align*}
\lim_{k\to\infty} &
	\int_\Omega \left[ \frac{\xi_k}{2}|\nabla \phi_k|^2  
                + \frac{W(\phi_k)}{\xi_k} \right] dx \nn\\
	&=\lim_{k\to\infty}\int_{\Omega} \left( \frac{ \sqrt{  W_a (\phi_k) +\xi_k } }{\sqrt 2} 
	+ \frac {a W_a (\phi_k)}{\sqrt {2 \left[ W_a (\phi_k) 
        +\xi_k \right] }} \right) |\nabla\phi_k|\, dx \nn	\\
	&=\lim_{k\to\infty}\int_0^1 \calH^{n-1}( \{ x\in \Omega: \phi_k (x) = s \} )
\left( \frac{\sqrt{  W_a (s) +\xi_k } }{\sqrt 2} 
	+ \frac {a W_a (s)}{\sqrt {2 \left[  W_a (s) +\xi_k \right] }} \right)\;ds \nn\\
	&=P_{\Omega}(G)\int_0^1 \frac{1+a}{\sqrt 2} \sqrt{ W_a (s)}\;ds	\nn\\
	&= \frac{1+a}{2\sqrt{a}}\sigma \, P_{\Omega}(G).
\end{align*}
If $\beta = \sigma$, we can take $a=1$. If $\beta>\sigma$, we have two choices of $a>0$ 
such that $\beta =  (1+a)\sigma /(2\sqrt{a}) $. Thus for any $\beta\geq \sigma$ we can find 
$\phi_k$ $(k = 1, 2, \dots)$ that satisfy (1) and (2). 
\section{The Poisson--Boltzmann Electrostatics}
\label{s:PB}


We first present some basic results regarding the boundary-value problem of Poisson--Boltzmann
equation and the corresponding electrostatic free energy for a function $\phi: \Omega \to \R$
that describes the dielectric environment. These results unify and improve those of 
Theorem~2.1 in \cite{LiChengZhang_SIAP11} and Theorem~2.1 in \cite{LiLiu_SIAP15}. 
We recall that the set $\calA$ and functional $E_\phi$ are defined in \reff{A}
and \reff{Ephiu}, respectively. 

\begin{theorem}
\label{t:phiPB}
Let $\phi \in L^4(\Omega).$ There exists a unique $\psi_{\phi} \in \calA$ such that
\begin{equation}
\label{Ephipsiphi}
 E_\phi [\psi_\phi] = \min_{u \in \calA} E_\phi [u], 
\end{equation}
which is finite. Moreover, $\psi_\phi \in \calA$ is the unique weak solution to
the boundary-value problem of Poisson--Boltzmann equation \reff{PhaseFieldPBE} and
\reff{PhaseFieldBC}, i.e., $\psi_{\phi} \in \calA$ and
\begin{equation}
\label{psiphiweak}
\int_\Omega  \left[ \ve(\phi) \nabla \psi_{\phi} \cdot \nabla \eta +
(\phi-1)^2 B'(\psi_{\phi} ) \, \eta \right] dx = \int_\Omega \rho  \eta \, dx
\qquad \forall \eta \in H_0^1(\Omega).
\end{equation}
Finally, $\psi_\phi \in L^\infty(\Omega)$ and 
there exists a constant $C > 0$ independent of $\phi \in L^4(\Omega)$ such that 
\begin{align*}
  & \|\chi_{\{ \phi \ne 1\} } \psi_{\phi} \|_{L^\infty(\Omega)} \le  C, \\
  & \|\psi_{\phi} \|_{H^1(\Omega)} \le C \left( 1 + \| \phi \|_{L^2(\Omega)} \right), \\
  & \|\psi_{\phi} \|_{L^\infty(\Omega)} \le  C \left( 1 + \| \phi \|_{L^4(\Omega)}^2 \right).  
\end{align*}
\end{theorem}

\begin{proof}
This is similar to that of Theorem~2.1 in \cite{LiLiu_SIAP15}. 
First, note that $B \in C^2(\R)$ is convex and nonnegative. By 
direct methods in the calculus of variations, there exists a unique $\psi_\phi \in \calA$ 
that satisfies \reff{Ephipsiphi}. The minimum value is finite as it is bounded above by
$E_\phi [\psi_\infty]< \infty.$ Next, by a comparison argument using the 
growth property and convexity of $B$ (cf.\ the proof of Theorem~2.1 in \cite{LiLiu_SIAP15}), 
we have $ | \psi_\phi |  \le C $ a.e.\ on $ {\{\phi \ne  1\}} $ for some constant
$C > 0$ independent of $\phi.$ This is the first desired estimate. 
This estimate, together with the Lebesgue Dominated Convergence
Theorem, allows us to obtain \reff{psiphiweak} for $\eta\in H_0^1(\Omega)\cap L^\infty(\Omega).$
By approximation, \reff{psiphiweak} is true for all $\eta \in H^1_0(\Omega).$ 
Finally, the fact that $\psi_\phi \in L^\infty(\Omega)$ and 
the other two desired estimates follow from the regularity theory for elliptic problems;
cf.\ Theorem 8.3 and Theorem 8.16 in \cite{GilbargTrudinger83}, and 
the proof of Theorem~2.1 in \cite{LiLiu_SIAP15}. In particular, the estimate 
(10) in \cite{LiLiu_SIAP15} provides the bound $C (1 + \| \phi \|^2_{L^4(\Omega)})$
for $\| \psi_\phi \|_{L^\infty(\Omega)}.$ 
\end{proof}

The following theorem indicates that the electrostatic potential and 
electrostatic free energy are continuous with respect to the change of dielectric boundary:  

\begin{theorem}
\label{t:PBenergy}
Let $\phi_k \in L^4(\Omega)$ $(k  = 1, 2, \dots)$ and $\phi \in L^4(\Omega) $ be such that 
\begin{equation}
\label{L4L1}
\sup_{k \ge 1} \| \phi_k \|_{L^4(\Omega)} < \infty \quad
\mbox{and} \quad  
\phi_k \to \phi  \quad \mbox{in } L^1(\Omega). 
\end{equation}
Let $\psi_{\phi_k}\in \calA$ $(k = 1, 2, \dots)$ and $\psi_\phi \in \calA$
be the corresponding electrostatic potentials, i.e., 
\[
E_{\phi_k} [\psi_{\phi_k} ] = \min_{u \in \calA} E_{\phi_k} [ u ]  \quad  (k = 1, 2, \dots)
\quad \mbox{and} \quad 
E_{\phi} [\psi_\phi ] = \min_{u \in \calA} E_{\phi} [ u ], 
\]
respectively.  Then, $\psi_{\phi_k} \to \psi_\phi$ in $H^1(\Omega)$
and $E_{\phi_k} [\psi_{\phi_k}] \to E_\phi [\psi_\phi].$
\end{theorem}

To prove this and other theorems, we need the following lemma which
holds true for any measurable subset $\Omega \subset \R^n$ of finite measure $|\Omega|:$  


\begin{lemma}
\label{l:Lq}
Let $1 < p < \infty$ and $\phi_k \in L^p(\Omega)$ $(k = 1, 2, \dots)$ be such that 
\begin{equation}
\label{pnormbound}
  \sup_{k \ge 1} \| \phi_k \|_{L^p(\Omega)} < \infty. 
\end{equation}
Let $\phi \in L^1(\Omega)$. Assume either $\phi_k \to \phi $ a.e.\ in $\Omega$ or 
$ \phi_k \to \phi$ in $L^1(\Omega).$ 
Then $\phi \in L^p(\Omega)$ and $\phi_k \to \phi$ in $L^q(\Omega)$ for any $q \in [1, p).$ 
\end{lemma} 


\begin{proof}
Assume $\phi_k \to \phi$ a.e.\ in $\Omega$. Fatou's lemma then leads to  
\[
\int_\Omega |\phi|^p dx \le \liminf_{k\to\infty} \int_\Omega |\phi_{k}|^p dx < \infty. 
\]
Hence  $\phi \in L^p(\Omega).$ 
Let $\ve > 0.$ Egoroff's Theorem implies that there exists a measurable subset $A \subseteq \Omega$
such that $|A| < \ve$ and $\phi_k \to \phi$ uniformly on $A^c = \Omega \setminus A $. 
Therefore, it follows from  H\"older's inequality and \reff{pnormbound} that for any $q \in [1, p)$
\begin{align*}
\limsup_{k\to \infty} \int_\Omega | \phi_k - \phi |^q dx 
& = \limsup_{k\to \infty} \left[ \int_{A} | \phi_k - \phi |^q dx + 
\int_{A^c} | \phi_k - \phi |^q dx \right] 
\\
& \le \limsup_{k\to \infty}  |A|^{(p-q)/p} \| \phi_k \|_{L^p(\Omega)}^q 
+ \limsup_{k \to \infty} \int_{A^c } | \phi_k - \phi |^q dx\\
& \le \ve^{(p-q)/p} \left( \sup_{k \ge 1} \| \phi_k \|_{L^p(\Omega)}^q \right).  
\end{align*}
Hence $\phi_k \to \phi$ in $L^q(\Omega).$ 

Assume now $\phi_k \to \phi$ in $L^1(\Omega)$. Then there exists a subsequence of $\{ \phi_k \}$
that converges to $\phi$ a.e.\ in $\Omega$. Applying Fatou's lemma to this subsequence, 
we also get $ \phi \in L^p(\Omega).$ 
Let $1 < q < p.$ Every subsequence of
$\{ \phi_k \}$ has a further subsequence that converges to $\phi$  a.e.\ in $\Omega$,
and hence, as proved above, converges to $\phi$ in $L^q(\Omega)$. 
 Thus $\phi_k\to \phi$ in $L^q(\Omega)$. 
\end{proof}

We are now ready to prove Theorem~\ref{t:PBenergy}.
We use the symbol $\rightharpoonup$ to denote the weak convergence:   

\begin{proof}[Proof of Theorem~\ref{t:PBenergy}]
For notational convenience, let us write $\psi_k = \psi_{\phi_k}$ and $\psi = \psi_\phi.$ 
We first prove that $\psi_k \to \psi$ in $H^1(\Omega)$. It suffices to prove that 
any subsequence of $\{ \psi_k \}$ has a further subsequence that converges to $\psi$
in $H^1(\Omega).$ 

Note by Theorem~\ref{t:phiPB} and \reff{L4L1} that 
\begin{align}
\label{PBk}
&\int_\Omega \left[ \ve(\phi_k) \nabla \psi_{k} \cdot\nabla \eta  
+ ( \phi_k - 1)^2 B'(\psi_{k}) \eta \right] dx 
= \int_\Omega \rho \eta \, dx  \qquad \forall \eta \in H^1_0(\Omega)
\quad \forall k \ge 1, 
\\
\label{PBphi}
&\int_\Omega \left[ \ve(\phi) \nabla \psi  \cdot\nabla \eta  
+ ( \phi - 1)^2 B'(\psi) \eta \right] dx 
= \int_\Omega \rho \eta \, dx  \qquad \forall \eta \in H^1_0(\Omega), 
\\
\label{supsup}
& \sup_{k \ge 1} \left( \| \psi_{k} \|_{H^1(\Omega)} +\| \psi_{k} \|_{L^\infty(\Omega)}\right)<\infty  
\quad \mbox{and} \quad \psi_\phi \in L^\infty(\Omega). 
\end{align}
By \reff{L4L1} and \reff{supsup}, any subsequence of $\{\psi_k\}$ has a further subsequence 
$\{\psi_{k_j} \}$ that converges to some $\hat{\psi} \in H^1(\Omega)$ 
weakly in $H^1(\Omega)$, strongly in $L^2(\Omega),$ and a.e.\ in $\Omega$; 
and the corresponding sequence $\{ \phi_{k_j} \}$ converges to $\phi$ a.e.\ in $\Omega$.  
We prove that $\hat{\psi} = \psi $ in $H^1(\Omega)$ and $\psi_{k_j} \to \psi$ strongly 
in $H^1(\Omega).$  


Since $\calA$ is convex and strongly closed in $H^1(\Omega),$  it is sequentially 
weakly closed. Hence $\hat{\psi} \in \calA.$  
Since $\psi_{k_j} \to \hat{\psi}$ a.e.\ in $\Omega$, 
 by \reff{supsup}, $\hat{\psi} \in L^\infty(\Omega).$ 
By Lemma~\ref{l:Lq}, $\phi_{k_j} \to \phi$  in $L^q(\Omega)$ for any $q \in [1, 4).$
Hence, $\ve(\phi_{k_j}) \to \ve( \phi ) $ in $L^2(\Omega).$ 
Similarly, 
\begin{equation}
\label{L3over2}
 (\phi_{k_j} -1)^2 \to (\phi - 1 )^2  \quad \mbox{in } L^{3/2}(\Omega). 
\end{equation}
By the compact embedding $H^1(\Omega) \hookrightarrow L^3(\Omega)$
and the weak convergence $\psi_{k_j} \rightharpoonup \hat{\psi}$ in $H^1(\Omega)$, 
we have that $\psi_{k_j} \to \hat{\psi}$ in $L^3(\Omega),$ and hence that 
\begin{equation}
\label{BprimeL3}
B'(\psi_{k_j}) \to B'(\hat{\psi}) \quad \mbox{in } L^3(\Omega). 
\end{equation}
Therefore, replacing $\phi_k$ and $\psi_k$ in \reff{PBk} by $\phi_{k_j}$ and $\psi_{k_j}$, 
respectively, and then sending $j \to \infty$, we obtain for any $\eta \in C_c^1(\Omega)$  that 
\[
\int_\Omega \left[ \ve(\phi) \nabla \hat{\psi} \cdot\nabla \eta  
+ (\phi-1)^2 B'(\hat{\psi} ) \eta \right] dx = \int_\Omega \rho\eta \, dx. 
\]
Since $C_c^1(\Omega)$ is dense in $H_0^1(\Omega),$
this identity holds true also for any $\eta \in H_0^1(\Omega)$. 
This and \reff{PBphi}, together with the uniqueness of weak solution established in 
Theorem~\ref{t:phiPB}, imply that  $\hat{\psi} =\psi$ in $H^1(\Omega).$ 

We now prove $\psi_{k_j} \to \psi$ in $H^1(\Omega).$ 
By our assumptions on $\ve$, the fact that $\psi_{k_j} - \psi\in H_0^1(\Omega)$
$(j = 1, 2, \dots)$,  and Poincar\'e's inequality, it suffices to prove 
\begin{equation}
\label{H1strong}
\lim_{j \to \infty} \int_\Omega \ve(\phi_{k_j } ) | \nabla \psi_{k_j}  - \nabla \psi|^2 dx = 0. 
\end{equation}

By  \reff{L4L1} and Lemma~\ref{l:Lq}, we have $\phi_{k_j} \to \phi$ in $L^{7/2}(\Omega)$
and hence $(\phi_{k_j} - 1)^2 \to (\phi - 1)^2$ in $L^{7/4}(\Omega).$ 
Similarly, by the convergence $\psi_{k_j} \to \psi$ in $L^2(\Omega)$, 
the embedding $H^1(\Omega) \hookrightarrow L^{14/3}(\Omega)$, \reff{supsup}, 
and Lemma~\ref{l:Lq}, we have $\psi_{k_j} \to \psi$  and hence $B(\psi_{k_j}) \to B(\psi)$
in $L^{14/3}(\Omega).$ Consequently, by H\"older's inequality, 
\begin{equation*}
\lim_{j\to \infty} \int_\Omega (\phi_{k_j} - 1 )^2 B(\psi_{k_j}) 
( \psi_{k_j} - \psi_\infty )  \, dx 
= \int_\Omega (\phi - 1 )^2 B(\psi) ( \psi - \psi_\infty) \, dx. 
\end{equation*}
Setting $\eta = \psi_{k_j} - \psi_\infty \in H_0^1(\Omega)$ in \reff{PBk} and \reff{PBphi}, 
we then obtain 
\begin{align}
\label{psikpsik}
& \lim_{j\to \infty} \int_\Omega \ve(\phi_{k_j})  |\nabla \psi_{k_j} |^2 dx 
\nonumber \\
& \qquad = \lim_{j \to \infty} 
\int_\Omega \left[ \ve(\phi_{k_j})  \nabla \psi_{k_j} \cdot \nabla \psi_\infty 
+  \ve(\phi_{k_j}) \nabla \psi_{k_j} \cdot \nabla (\psi_{k_j} - \psi_\infty) \right] dx 
\nonumber \\
& \qquad = \lim_{j \to \infty} 
\int_\Omega \left[ \ve(\phi_{k_j}) \nabla \psi_{k_j} \cdot \nabla \psi_\infty 
+ \rho ( \psi_{k_j}- \psi_\infty ) - (\phi_{k_j} -1)^2 B'(\psi_{k_j})  
( \psi_{k_j} - \psi_\infty ) \right]  dx 
\nonumber \\
& \qquad = \int_\Omega \left[ \ve(\phi) \nabla \psi \cdot \nabla \psi_\infty 
+\rho ( \psi - \psi_\infty ) - (\phi-1)^2 B'(\psi)  ( \psi - \psi_\infty ) \right]  dx \nonumber \\
& \qquad 
= \int_\Omega \left[ \ve(\phi)  \nabla \psi \cdot \nabla \psi_\infty 
+  \ve(\phi) \nabla \psi \cdot \nabla (\psi - \psi_\infty) \right] dx 
\nonumber \\
&\qquad  = \int_\Omega \ve(\phi) |\nabla \psi |^2 dx.  
\end{align} 
Since $\phi_{k_j} \to \phi$ a.e.\ in $\Omega,$ the Lebesgue Dominated Convergence Theorem implies that
\begin{equation}
\label{vekpsi2}
\lim_{j \to \infty} \int_\Omega \ve(\phi_{k_j}) | \nabla \psi|^2 dx 
= \int_\Omega \ve(\phi) | \nabla \psi|^2 dx. 
\end{equation}
It now follows from \reff{psikpsik}, \reff{vekpsi2}, and the fact that 
$\ve(\phi_{k_j}) \to \ve(\phi)$ in $L^2(\Omega)$ and 
$\psi_{k_j} \rightharpoonup \psi$ in $H^1(\Omega)$ that 
\begin{align*}
& \lim_{j \to \infty} \int_\Omega \ve(\phi_{k_j}) | \nabla \psi_{k_j} - \nabla \psi|^2 dx 
\\
&\quad 
= \lim_{j \to \infty} \int_\Omega \left[ \ve(\phi_{k_j}) | \nabla \psi_{k_j} |^2 
- 2 \ve(\phi_{k_j}) \nabla \psi_{k_j} \cdot \nabla \psi + \ve(\phi_{k_j}) | \nabla \psi |^2 \right] dx
\\
&\quad 
= \int_\Omega \left[ \ve(\phi) | \nabla \psi |^2 - 2 \ve(\phi) \nabla \psi \cdot \nabla \psi
+ \ve(\phi) | \nabla \psi |^2 \right] dx
\\
& \quad 
= 0,  
\end{align*}
leading to \reff{H1strong}. 

We finally prove the energy convergence $E_{\phi_k}[\psi_k] \to E_\phi [\psi].$ 
Since $\phi_k \to \phi$ in $L^1(\Omega)$ and $\psi_k \to \psi$ in $H^1(\Omega)$, 
 any subsequence of $\{ \phi_k \}$ and the corresponding
subsequence of $\{ \psi_k \}$ have further subsequneces $\{ \phi_{k_j} \}$ and 
$\{ \psi_{k_j} \}$, respectively, such that $\phi_{k_j} \to \phi $ a.e.\ in $\Omega$, 
and $\psi_{k_j} \to \psi$ in $H^1(\Omega)$ and a.e.\ in $\Omega.$ 
By \reff{L3over2}, and \reff{BprimeL3} with $\psi$ replacing $\hat{\psi}$, we have
\begin{equation}
\label{fB}
\lim_{k \to \infty}
\int_\Omega \left[ -\rho \psi_{k_j} + (\phi_{k_j} -1)^2 B(\psi_{k_j} ) \right] dx
= \int_\Omega \left[ -\rho \psi + (\phi-1)^2 B(\psi) \right] dx. 
\end{equation}
This and \reff{psikpsik}  implies that $ E_{\phi_{k_j}} [\psi_{k_j}] \to  E_{\phi}[\psi]$. 
Hence $E_{\phi_k}[\psi_k] \to E_\phi [\psi].$ 
\end{proof}

We now state and prove the last result in this section: the convergence to the sharp-interface limit of
phase-field electrostatic boundary forces, in terms of the weak convergence 
of the corresponding stress fields; cf.~Lemma~\ref{l:Stress}. 
We recall that $ f_{0,{\rm ele}}[\partial G]$ is defined in \reff{f0_ele}.

\begin{theorem}[Convergence of dielectric boundary force]
\label{th:f_ele-conv}
Let $\phi_k \in L^4(\Omega)$ $(k = 1, 2, \dots)$ and $\phi \in L^1(\Omega)$ be such that 
\begin{equation}
\label{threephik}
\sup_{k \ge 1} \| \phi_k \|_{L^4(\Omega)} < \infty
\quad  \mbox{and}  \quad 
\phi_k \to \phi  \quad \mbox{a.e.\ in } \Omega. 
\end{equation}
We have for any $ V \in C_c^1(\Omega, \R^3)$ that 
\begin{align}\label{Fgeneral}
	\lim_{k\to \infty} \int_\Omega \bigl[  T_{\rm ele}(\phi_k) 
	: \nabla V -\rho \nabla \psi_{\phi_k} \cdot V \bigr]   dx 
	= \int_\Omega \bigl[ T_{\rm ele}(\phi) : \nabla V -\rho \nabla \psi_{\phi} \cdot V \bigr] dx. 
\end{align}
If, in addition,  $\phi = \chi_G$ for some open subset $G$  of $\Omega$ with a $C^2$
boundary $\partial G$ and the closure $\overline{G} \subset \Omega$, then this limit is 
\begin{equation}
\label{Fspecial}
\int_\Omega \bigl[  T_{\rm ele}(\chi_G ) : \nabla V -\rho \nabla \psi_{\chi_G} \cdot V \bigr] dx  
= - \int_\Omega f_{0, {\rm ele}}[\partial G] \cdot V \, dS.    
\end{equation}

\end{theorem}

\begin{proof}
We first note that, by Lemma~\ref{l:Lq}, $\phi \in L^4(\Omega)$ and $\phi_k \to \phi$
in $L^q(\Omega)$ for any $q\in [1, 4).$  Let us denote 
$\psi_k = \psi_{\phi_k}$ $(k \ge 1)$ and $\psi = \psi_\phi.$ Since $\ve$ is a bounded 
function and $\psi_k \to \psi$ in $H^1(\Omega)$ by Theorem~\ref{t:PBenergy}, we have 
\begin{align*}
&\lim_{k \to \infty}  \int_\Omega    \ve(\phi_k) 
\left[ ( \nabla \psi_k - \nabla \psi)  \otimes ( \nabla \psi_k - \nabla \psi)
\right.
\\
&\qquad
\left.   + 
 \nabla \psi \otimes ( \nabla \psi_k - \nabla \psi) + 
( \nabla \psi_k  - \nabla \psi) \otimes  \nabla \psi \right] : \nabla V \, dx  = 0. 
\end{align*}
Since $\phi_k \to \phi$ a.e.\ in $\Omega$, the Lebesgue Dominated Convergence Theorem
implies that 
\[
\lim_{k \to \infty} \int_\Omega 
 \ve(\phi_k) \nabla \psi  \otimes \nabla \psi :  \nabla V \, dx
= \int_\Omega 
 \ve(\phi) \nabla \psi \otimes \nabla \psi : \nabla V \, dx. 
\]
Therefore, 
\begin{align}
\label{q1}
&\lim_{k\to \infty} \int_\Omega \ve(\phi_k) \nabla \psi_k \otimes \nabla \psi_k : \nabla V \, dx  
\nonumber \\
& \qquad 
= \lim_{k \to \infty}  \int_\Omega    \ve(\phi_k) 
\left[ ( \nabla \psi_k - \nabla \psi)  \otimes ( \nabla \psi_k - \nabla \psi)
 + \nabla \psi \otimes ( \nabla \psi_k - \nabla \psi)  
\right.
\nonumber \\
&\qquad \qquad
\left.   + 
( \nabla \psi_k  - \nabla \psi) \otimes  \nabla \psi + \nabla \psi
\otimes \nabla \psi \right] : \nabla V \, dx  
\nonumber \\
&\qquad 
= \int_\Omega \ve(\phi) \nabla \psi \otimes \nabla \psi : V \, dx. 
\end{align}
Similarly, 
\begin{equation}
\label{q2}
\lim_{k\to \infty} \int_\Omega \ve(\phi_k) |\nabla \psi_k|^2 \nabla \cdot V \, dx 
= \int_\Omega \ve(\phi) |\nabla \psi|^2 \nabla \cdot V \, dx.  
\end{equation}
As in the proof of Theorem~\ref{t:PBenergy}, we have again by the convergence 
$\psi_k \to \psi$ in $H^1(\Omega)$ that 
\begin{align}
\label{q3}
& \lim_{k \to \infty} 
\int_\Omega \left[ (\phi_k-1)^2 B(\psi_k) \nabla \cdot V +\rho \nabla \psi_k \cdot V \right] \, dx
= \int_\Omega \left[ (\phi-1)^2 B(\psi) \nabla \cdot V + \rho \nabla \psi \cdot V \right] \, dx. 
\end{align}
It now follows from the definition of $T_{\rm ele}$ (cf.\ \reff{stressphielec})
and \reff{q1}--\reff{q3} that 
\begin{align*}
&\lim_{k\to \infty} \int_\Omega \bigl[ T_{\rm ele}(\phi_k) : \nabla V - \rho \nabla \psi_{\phi_k} 
\cdot V \bigr]  dx 
\\
&\qquad 
= \lim_{k\to \infty} 
\int_\Omega \biggl\{  \ve(\phi_k) \nabla \psi_k \otimes \nabla \psi_k : \nabla V 
- \left[ \frac12 \ve(\phi_k) | \nabla \psi_k |^2 + (\phi_k-1)^2 B(\psi_k) \right] \nabla \cdot V
\\
&\qquad \qquad \quad 
- \rho \nabla \psi_k \cdot V \biggr\}  dx
\\
&\qquad 
= \int_\Omega \biggl\{   \ve(\phi) \nabla \psi \otimes \nabla \psi : \nabla V 
- \left[  \frac12 \ve(\phi) | \nabla \psi |^2 + (\phi-1)^2 B(\psi) \right]  \nabla \cdot V
- \rho \nabla \psi \cdot V \biggr\}   dx
\\
&\qquad 
= \int_\Omega \bigl[  T_{\rm ele}(\phi) : \nabla V - \rho \nabla \psi \cdot V \bigr] dx.  
\end{align*}
This is exactly \reff{Fgeneral}, since $\psi = \psi_\phi.$  

We now prove \reff{Fspecial}. Denote again $\psi = \psi_\phi = \psi_{\chi_G} \in \calA.$ 
Denote also by  $V_i$ and $\nu_i$ $(i = 1, 2, 3)$ the components of $V$ and
$\nu$, respectively. Notice that the unit normal $\nu$ 
points from $G $ to $G^c=\Omega \setminus G$.
Using the conventional summation notation, we have by integration by parts that 
\begin{align}
\label{DBFlong}
&\int_\Omega \left[ T_{\rm ele}(\chi_G) : \nabla V  -\rho \nabla \psi_{\chi_G} \cdot V  
\right]  dx 
\nonumber 
\\
& \qquad 
= \int_\Omega \left\{  \ve(\chi_G) \nabla \psi \otimes \nabla \psi : \nabla V
- \left[ \frac{\ve(\chi_G)}{2} | \nabla \psi |^2 + \chi_{G^c} B(\psi) \right]  
 \nabla \cdot V -\rho \nabla \psi \cdot V \right\} dx
\nonumber 
\\
& \qquad 
= \int_G  \left(  \ve_{\rm p} \partial_i  \psi \partial_j \psi  \partial_j V_i
- \frac{\ve_{\rm p}}{2} \partial_i \psi \partial_i \psi \partial_j V_j 
-\rho  \nabla \psi \cdot V  \right)  dx
\nonumber 
\\
&\qquad \qquad 
+  \int_{G^c}  \left[ \ve_{\rm w} \partial_i  \psi \partial_j \psi  \partial_j V_i
- \frac{\ve_{\rm w}}{2} \partial_i \psi \partial_i \psi \partial_j V_j 
	- B(\psi) \partial_j V_j -\rho  \nabla \psi \cdot V   \right] dx
\nonumber 
\\
& \qquad 
= \int_G  \left( -\ve_{\rm p} \partial_{ij}  \psi \partial_j \psi  V_i
 -\ve_{\rm p} \partial_i \psi \partial_{jj} \psi  V_i
+ \ve_{\rm p} \partial_{ij}  \psi \partial_i \psi  V_j -\rho \nabla \psi \cdot V \right)  \, dx
\nonumber 
\\
&\qquad \qquad 
+ \int_{\partial G} \left(   \ve_{\rm p} \partial_i \psi|_G \partial_j \psi|_G V_i \nu_j 
- \frac{\ve_{\rm p}}{2} \partial_i \psi|_G \partial_i \psi|_G V_j \nu_j \right)  dS
\nonumber 
\\
&\qquad \qquad 
+ \int_{G^c}  \left[ - \ve_{\rm w} \partial_{ij}  \psi \partial_j \psi  V_i
- \ve_{\rm w} \partial_i \psi \partial_{jj} \psi V_i
+ \ve_{\rm w} \partial_{ij}  \psi \partial_i \psi  V_j 
+B'(\psi) \partial_j \psi V_j -\rho \nabla \psi \cdot V  \right] \, dx
\nonumber 
\\
&\qquad \qquad 
+ \int_{\partial G} \left[ - \ve_{\rm w} \partial_i \psi|_{G^c} \partial_j \psi|_{G^c} V_i \nu_j 
+ \frac{\ve_{\rm w}}{2} \partial_i \psi|_{G^c} \partial_i \psi|_{G^c} V_j \nu_j 
+ B(\psi) V_j \nu_j \right] dS
\nonumber 
\\
&\qquad 
= \int_G  ( - \ve_{\rm p} \Delta \psi -\rho ) \nabla \psi \cdot V \, dx
+  \int_{G^c}   \left[ -\ve_{\rm w} \Delta \psi + B'(\psi) -\rho \right] \nabla \psi \cdot V \, dx
\nonumber 
\\
&\qquad \qquad 
+ \int_{\partial G} \biggl\{   
\ve_p ( \nabla \psi \cdot \nu ) \nabla \psi|_{G} \cdot V 
- \ve_{\rm w}  ( \nabla \psi \cdot \nu ) \nabla \psi|_{G^c} \cdot V 
\nonumber 
\\
&\qquad \qquad \qquad
+ \left[  \frac{\ve_{\rm w}}{2} | \nabla \psi|_{G^c} |^2 
- \frac{\ve_{\rm p}}{2} | \nabla \psi|_{G} |^2  + B(\psi)  \right] V \cdot \nu  \biggr\}   dS
\nonumber 
\\
&\qquad 
= \int_{\partial G} \biggl\{  \ve(\chi_G)  (\nabla \psi \cdot \nu ) 
(\nabla \psi|_{G} - \nabla \psi|_{G^c} ) \cdot V
\nonumber 
\\
&\qquad \qquad 
+ \left[  \frac{\ve_{\rm w}}{2} | \nabla \psi|_{G^c} |^2 
- \frac{\ve_{\rm p}}{2} | \nabla \psi|_{G} |^2  + B(\psi)  \right] V \cdot \nu  \biggr\} dS, 
\end{align}
where in the last step we used \reff{psiGp}--\reff{psiGgrad}. 

The gradient $\nabla \psi$ restricted onto $\partial G$ 
from either $G$ or $G^c$ has the decomposition
\[
\nabla \psi = (\nabla \psi \cdot \nu ) \nu + (I - \nu \otimes \nu)
\nabla \psi \qquad \mbox{on } \partial G.
\]
Since $\psi$ is continuous across $\partial G$ (cf.\ \reff{psiGcont}), 
the tangential derivatives of $\psi$, and hence $(I-\nu \otimes \nu) \nabla \psi,$
are continuous across the interface $\partial G$: 
\[
(I - \nu\otimes\nu) \nabla \psi|_G = (I - \nu\otimes\nu) \nabla \psi|_{G^c} 
\qquad \mbox{on } \partial G. 
\]
Thus
\[
\nabla \psi|_{G} - \nabla \psi|_{G^c} 
= (( \nabla \psi|_{G} - \nabla \psi|_{G^c} ) \cdot \nu ) \nu
 \qquad \mbox{on } \partial G.   
\]
Moreover, restricted onto $\partial G$ from either $G$ or $G^c$, 
\[
|\nabla \psi|^2 = | (\nabla \psi \cdot \nu ) \nu  + (I - \nu\otimes\nu) \nabla \psi|^2 = 
 | \nabla \psi \cdot \nu |^2  + |(I - \nu\otimes\nu) \nabla \psi|^2. 
\]
Therefore, 
\begin{align*}
&  \ve(\chi_G)  ( \nabla \psi \cdot \nu ) (\nabla \psi|_{G} - \nabla \psi|_{G^c} ) \cdot V
+ \left[  \frac{\ve_{\rm w}}{2} | \nabla \psi|_{G^c} |^2 
- \frac{\ve_{\rm p}}{2} | \nabla \psi|_{G} |^2  + B(\psi)  \right]  V \cdot \nu 
\\
&\qquad 
= \left[ \ve_{\rm p} | \nabla \psi|_{G }  \cdot \nu |^2 
- \ve_{\rm w} | \nabla \psi|_{G^{c}}  \cdot \nu |^2 + \frac{\ve_{\rm w}}{2} | \nabla \psi|_{G^c} |^2 
- \frac{\ve_{\rm p}}{2} | \nabla \psi|_{G} |^2 + B(\psi) \right]   V \cdot \nu 
\\
&\qquad 
= \left[ \frac{\ve_{\rm p}}{2}  | \nabla \psi|_{G }  \cdot \nu |^2 
 - \frac{\ve_{\rm w}}{2}  | \nabla \psi|_{G^c}  \cdot \nu |^2 +  \frac12 ( \ve_{\rm w} - \ve_{\rm p} ) 
| (I - \nu\otimes\nu) \nabla \psi |^2 + B(\psi) \right]   V \cdot \nu 
\\
&\qquad 
= \left[ \frac12 \left( \frac{1}{\ve_{\rm p}}  - \frac{1}{\ve_{\rm w}} \right)
 | \ve(\chi_G) \nabla \psi \cdot \nu |^2 +  \frac12 ( \ve_{\rm w} - \ve_{\rm p} ) 
| (I - \nu\otimes\nu) \nabla \psi |^2 + B(\psi) \right]   V \cdot \nu 
\\
&\qquad 
= - f_{0, {\rm ele}} [\partial G] \cdot V. 
\end{align*}
With our notation $\psi = \psi_{\chi_G},$ this and \reff{DBFlong} imply \reff{Fspecial}. 
\end{proof}

\section{Free-Energy Convergence}
\label{s:FreeEnergyConvergence}


In this section, we first prove some lemmas. We then prove Theorem~\ref{t:EnergyConvergence}
on the $\Gamma$-convergence of free-energy functionals
and its Corollary~\ref{c:existenceF0}. Finally, we prove Theorem~\ref{t:individual} 
on the equivalence of the convergence of
total free energy and that of each individual part of the free energy. 

The first lemma is on the existence of a phase-field minimizer 
for the functional $F_\xi$ (cf.\ \reff{newFxiphi}) for each $\xi \in (0, \xi_0].$ This result will be used in proving Corollary~\ref{c:existenceF0}.


\begin{lemma}
\label{l:minimizerFxi}
Let $\xi \in (0, \xi_0].$ There exists $\phi_\xi \in H^1(\Omega)$ such that 
\[
F_\xi [\phi_\xi  ] = \min_{\phi \in H^1(\Omega)} F_\xi [\phi ] 
= \min_{\phi \in L^1(\Omega)} F_\xi [\phi ], 
\]
 which is finite. 
\end{lemma}

\begin{proof}
Let $\phi \in H^1(\Omega).$ We have by our assumptions on the functions $U$ and $\ve$, the fact that 
\[
W(s)-s^4 = 18s^2(s-1)^2 - s^4 \to +\infty \quad \mbox{as } s \to \infty, 
\]
the inequality 
\[
\min_{u\in \calA} E_\phi [u] \le E_\phi [\psi_\infty]
= \int_\Omega \left[ \frac{\ve (\phi) }{2} |\nabla \psi_\infty  |^2
- \rho \psi_\infty   + (\phi-1)^2 B( \psi_\infty )  \right] dx, 
\]
and H\"older's inequality that 
\begin{align}
\label{boundxi}
F_\xi [\phi ] & \ge \int_\Omega \left[ P_0\phi^2 + \frac{\gamma_0 \xi}{2} |\nabla \phi|^2 \right] dx
+ \frac{\gamma_0}{ \xi} \|\phi  \|_{L^4(\Omega)}^4
+ \frac{\gamma_0}{ \xi} \int_\Omega \left[ W(\phi ) - \phi^4 \right] dx 
\nonumber 
\\
&\qquad + \rho_0  \int_{\{ x \in \Omega: U(x) \le 0 \}} (\phi  -1)^2 U\, dx - E_\phi [\psi_\infty]
\nonumber 
\\
&\ge C_1  \left( \| \phi \|_{H^1(\Omega)}^2 + \|\phi \|^4_{L^4(\Omega)} \right)
- 2 \left( \rho_0  |U_{\rm min}| + \| B(\psi_\infty) \|_{L^\infty(\Omega)} \right)
  \int_{\Omega} \phi^2 \, dx - C_2  
\nonumber \\
&\ge C_3 \left( \| \phi \|_{H^1(\Omega)}^2 + \|\phi \|^4_{L^4(\Omega)} \right) - C_4, 
\end{align}
where all $C_i $ $(i = 1, \dots, 4)$  are positive constants independent of $\phi \in H^1(\Omega).$

Let $\alpha = \inf_{\phi \in H^1(\Omega)} F_\xi[\phi ].$  By \reff{boundxi}, $\alpha > -\infty.$
Setting $\phi (x) = 1$ for all $x \in \Omega$, we have $\alpha \le E_\xi [\phi ] < \infty.$ 
So, $\alpha $ is finite. 
Let $\phi_k \in H^1(\Omega)$ $(k = 1, 2, \dots)$ be such that $F_\xi[\phi_k] \to \alpha.$ 
By \reff{boundxi}, $\{ \phi_k \}$ is bounded in $H^1(\Omega)$. Hence, it has a subsequence, 
not relabeled, such that $\psi_k \to \phi_\xi $ weakly in $H^1(\Omega)$, strongly in $L^2(\Omega)$, 
and a.e.\ in $\Omega$ for some $\phi_\xi  \in H^1(\Omega).$  

Since $\phi_k \to \phi_\xi $ in $L^2(\Omega)$ and $U$ is bounded below, 
\begin{align}
\label{L2phik}
& \lim_{k\to \infty } \left[ P_0\int_\Omega \phi_k^2 \, dx  + \rho_0  
\int_{\{ x\in\Omega: U(x) \le 0 \}} (\phi_k-1)^2 U \, dx \right]  
\nonumber \\
& \quad 
= P_0 \int_\Omega \phi_\xi^2 \, dx  + \rho_0 \int_{\{ x\in\Omega: U(x) \le 0 \}} 
(\phi_\xi -1)^2 U \, dx.  
\end{align}
Since $\phi_k \to \phi_\xi $ weakly in $H^1(\Omega)$, 
\begin{equation}
\label{H1weak}
 \liminf_{k\to \infty}  \gamma_0 \int_\Omega  \frac{\xi }{2} |\nabla \phi_k |^2 dx  
\ge \gamma_0 \int_\Omega  \frac{\xi }{2} |\nabla \phi_{\xi} |^2 dx. 
\end{equation} 
Since $\phi_k \to \phi_\xi$ a.e.\ in $\Omega$, Fatou's Lemma implies that 
\begin{align}
\label{Fatouxi}
& \liminf_{k\to \infty} \left[ \gamma_0 \int_\Omega \frac{1}{\xi} W(\phi_k) \, dx 
+  \rho_0 \int_{\{ x \in \Omega:U(x) > 0 \}} (\phi_k-1)^2 U\, dx \right] 
\nonumber \\
&\qquad \ge \gamma_0 \int_\Omega \frac{1}{\xi} W(\phi_\xi) \, dx + 
\rho_0  \int_{\{ x \in \Omega:U(x) > 0 \}} (\phi_\xi -1)^2 U\, dx. 
\end{align}
By the Sobolev embedding $H^1(\Omega) \hookrightarrow L^4(\Omega),$ 
 $\sup_{k \ge 1} \| \phi_k \|_{L^4(\Omega)} < \infty.$ 
Hence it follows from Theorem~\ref{t:PBenergy} that 
\begin{equation}
\label{elexi}
\lim_{k \to \infty} \min_{u \in \calA} E_{\phi_k}[u] = \min_{u \in \calA} E_{\phi_\xi}[u]. 
\end{equation}
Combining \reff{L2phik}--\reff{elexi}, we obtain 
\[
\alpha = \liminf_{k\to \infty} F_\xi[\phi_k]\ge F_\xi[\phi_\xi] \ge \alpha. 
\]
Hence $F_\xi[\phi_\xi] = \min_{\phi \in H^1(\Omega)} F_\xi [\phi ]$. 
But $F_\xi [\phi ] = +\infty$ if $\phi \in L^1(\Omega) \setminus 
H^1(\Omega).$ Hence 
$
F_\xi[\phi_\xi ] = \min_{\phi \in L^1(\Omega)} F_\xi [\phi]. 
$
\end{proof}

Next, we establish some lower bound for the functional $F_\xi = F_\xi [\phi]$ for 
all $\phi$ and $\xi.$ 

\begin{lemma}
\label{l:LowerBound}
There exists a constant $C$ such that 
for any $\phi \in H^1(\Omega)$ and any $\xi \in (0, \xi_0]$ 
\begin{equation}
\label{FxiphiLowerBound}
F_\xi [\phi] \ge \frac{\gamma_0 }{2}  \left[ \xi \| \nabla \phi \|_{L^2(\Omega)}^2 + 
\frac{1}{\xi}  \| W(\phi) \|_{L^1(\Omega)}\right] + 9 \gamma_0 \| \phi \|_{L^4(\Omega)}^4 
+ \rho_0  \int_\Omega (\phi-1)^2 | U | \, dx +C.
\end{equation}
\end{lemma}

\begin{proof}
Fix $\phi \in H^1(\Omega)$ and $\xi \in (0, \xi_0]$. Recall from \reff{Ephiu} that 
\[
E_\phi [ \psi_\infty ] = \int_\Omega \left[ \frac{\ve (\phi) }{2} |\nabla \psi_\infty  |^2
- \rho \psi_\infty   + (\phi-1)^2 B( \psi_\infty  )  \right] dx.
\]
We have then by the definition of $F_\xi$ (cf.\ \reff{newFxiphi}) that 
\begin{align}
\label{proofLowerBound}
0 &\le \frac{\gamma_0 }{2}  \left[ \xi \| \nabla \phi \|_{L^2(\Omega)}^2 + 
\frac{1}{\xi}  \| W(\phi) \|_{L^1(\Omega)}\right] + 9 \gamma_0  \| \phi \|_{L^4(\Omega)}^4 
+ \rho_0  \int_\Omega (\phi-1)^2 | U | \, dx  \nonumber \\
& =F_\xi [\phi] - P_0 \|\phi\|^2_{L^2(\Omega)}-\frac{\gamma_0}{2 \xi} \|W(\phi)\|_{L^1(\Omega)} 
+ 9 \gamma_0 \| \phi \|_{L^4(\Omega)}^4
\nonumber \\
& \qquad 
+ \rho_0 \int_\Omega (\phi-1)^2 ( |U| - U) \, dx + \min_{u \in \calA} E_\phi[u]
\nonumber \\
& \le  
F_\xi [\phi] - \frac{\gamma_0}{2 \xi_0} \|W(\phi)\|_{L^1(\Omega)} 
+ 9 \gamma_0  \| \phi \|_{L^4(\Omega)}^4
+ 2 \rho_0  \int_{\{ x \in \Omega: U(x) \le 0\}}  (\phi-1)^2  |U| \, dx + E_\phi[\psi_\infty]
\nonumber \\
&  \le  
F_\xi [\phi] - \frac{\gamma_0}{2 \xi_0} \|W(\phi)\|_{L^1(\Omega)} 
+ 9 \gamma_0 \| \phi \|_{L^4(\Omega)}^4
+ 2 \rho_0 | U_{\rm min} | \int_{\Omega} (\phi-1)^2  \, dx 
\nonumber \\
& \qquad 
+\frac{1}{2} \max (\ve_{\rm p},  \ve_{\rm w})
\| \nabla \psi_\infty \|^2_{L^2(\Omega)} + \| \rho\|_{L^2(\Omega)} \| \psi_\infty \|_{L^2(\Omega)}
+ \| B(\psi_\infty)\|_{L^\infty(\Omega)} \int_\Omega (\phi-1)^2 dx
\nonumber \\
& = F_\xi [\phi] - \int_\Omega g(\phi) \, dx 
+\frac{1}{2} \max (\ve_{\rm p},  \ve_{\rm w})
\| \nabla \psi_\infty \|^2_{L^2(\Omega)} + \| \rho \|_{L^2(\Omega)} \| \psi_\infty \|_{L^2(\Omega)},
\end{align}
where $g:\R\to\R$ is given by 
\[
g(s) = \frac{\gamma_0}{2 \xi_0} W(s) - 9 \gamma_0 s^4 - 
\left[  2 \rho_0 | U_{\rm min} |  + \|B(\psi_\infty)\|_{L^\infty(\Omega)} \right] (s-1)^2. 
\]
Note that $\lim_{s\to \infty} g(s) = +\infty$, since 
 $0 < \xi_0 < 1$ and $W(s) = 18s^2 (s-1)^2$. Therefore, $g$ is bounded below.
Setting 
\[
C = |\Omega| \min_{s \in \R} g(s) 
-\frac{1}{2} \max (\ve_{\rm p},  \ve_{\rm w})
\| \nabla \psi_\infty \|^2_{L^2(\Omega)} - \| \rho\|_{L^2(\Omega)} \| \psi_\infty \|_{L^2(\Omega)},  
\]
we then obtain the desired estimate \reff{FxiphiLowerBound} from \reff{proofLowerBound}. 
\end{proof}

The following lemma, stated for $\R^n$ with a general $n \ge 2$, is refinement of a standard result;  
it is used in the proof of Theorem~\ref{t:EnergyConvergence} and Theorem~\ref{th:CH-force-conv}:  

\begin{lemma}
\label{l:etak}
Let $\Omega$ be a nonempty, bounded, and open subset of $\R^n$ with $n \ge 2.$ 
Let $G$ be a measurable subset of $\Omega$ with $P_\Omega(G) < \infty.$ 
Assume that $\xi_k \searrow 0$ and  $\phi_k \in H^1(\Omega)$ $(k = 1, 2, \dots)$ satisfy
$\phi_k \to \chi_G$ a.e.\ in $\Omega$ and 
\begin{equation}
\label{supsupsup}
\sup_{k \ge 1} \int_\Omega \left[ \frac{\xi_k}{2} |\nabla \phi_k|^2 + \frac{1}{\xi_k}
W(\phi_k) \right] dx < \infty. 
\end{equation}
Define 
\[
\eta_k(x) = \int_0^{\phi_k(x)} \sqrt{2 W(t) } \, dt \qquad \forall  x \in \Omega, \,  k = 1, 2, \dots 
\]
Then 
\begin{align}
\label{etakbound}
& \sup_{k \ge 1} \left[ 
 \| \eta_k \|_{L^{4/3}(\Omega)} + \| \eta_k \|_{W^{1,1}(\Omega)} \right] < \infty, 
\\
\label{etakconv}
& \eta_k \to \chi_G  \ \mbox{a.e.\ in } \Omega \ \mbox{and} \
\mbox{in } L^q(\Omega) \ \mbox{for any } q \in [1, 4/3), 
\\
\label{POGetak}
& P_\Omega(G) \le \liminf_{k \to \infty} \int_\Omega | \nabla \eta_k | \, dx  
\le \liminf_{k\to \infty} 
 \int_\Omega \left[ \frac{\xi_k}{2} |\nabla \phi_k|^2 + \frac{1}{\xi_k}
W(\phi_k) \right] dx. 
\end{align} 
If, in addition, $\overline{G} \subset \Omega,$ then 
\begin{equation}
\label{etakweakconv}
\lim_{k\to \infty} \int_\Omega  \nabla \eta_k \cdot g \, dx =
-\int_{\partial^* G} g\cdot \nu \, d\calH^{n-1} \qquad \forall g \in C_c(\Omega, \R^n).    
\end{equation}
\end{lemma}

\begin{proof}
Since $W$ is a quartic potential, we have $\sqrt{2 W(t)}  \le C (1+ t^2) $ for all $ t \in \R. $   
Here and below, $C$ denotes a generic, positive constant. 
Therefore, 
\[
|\eta_k| \le  C( |\phi_k| + |\phi_k|^3) \qquad \mbox{a.e.\ in } \Omega, \, k = 1, 2, \dots 
\]
By \reff{supsupsup}, $\sup_{k \ge 1} \| \phi_k \|_{L^4(\Omega)} < \infty.$ 
This implies that 
\begin{equation}
\label{onesup}
\sup_{k \ge 1 } \|\eta_k \|_{L^{4/3}(\Omega)} < \infty. 
\end{equation}
Note for each $k \ge 1$ that $\nabla \eta_k  = \sqrt{2 W(\phi_k)} \nabla \phi_k$ a.e.\ in $\Omega.$ 
Hence, 
\begin{equation*}
\int_\Omega|\nabla \eta_k| \;dx= \int_\Omega \left|\sqrt{2 W(\phi_k)} \nabla \phi_k\right|\;dx 
\le \int_\Omega\left[ \frac{\xi_k}{2}|\nabla\phi_k|^2 
 + \frac{1}{\xi_k}W(\phi_k)\right] dx. 
\end{equation*} 
This, together with \reff{supsupsup} and \reff{onesup}, then implies that 
\begin{equation}
\label{twosup}
\sup_{k \ge 1 } \|\eta_k \|_{W^{1,1}(\Omega)} < \infty. 
\end{equation}
Now \reff{etakbound} follows from \reff{onesup} and \reff{twosup}. 

Since $\phi_k \to \chi_G$ a.e.\ in $\Omega$
and the integral of $\sqrt{2 W(s)}$ over $[0,1]$ is $1$, we have 
$\eta_k \to \chi_G$ a.e.\ in $\Omega$. 
Lemma~\ref{l:Lq} and \reff{onesup} imply that $\eta_k \to \chi_G$ in 
$L^q(\Omega)$ for any $q \in [1, 4/3).$ Hence \reff{etakconv} is proved. 

By the fact that $W^{1,1}(\Omega) \hookrightarrow BV(\Omega)$ and \reff{etakbound}, we have 
$\sup_{k\ge 1}\|\eta_k\|_{\rm BV(\Omega)} <\infty.$
Consequently, by \reff{etakconv} 
 \cite{Giusti84,Ziemer_Book89,EvansGariepy_Book92}, 
\begin{align*}
P_\Omega(G)  &\le \liminf_{k\to \infty} \int_\Omega |\nabla \eta_k | \, dx 
\\
& = \liminf_{k\to \infty} \int_\Omega \sqrt{2 W(\phi_k)} |\nabla \phi_k| \, dx 
\\
& \le \liminf_{k\to\infty} \int_\Omega \left[ \frac{\xi_k}{2} |\nabla \phi_k|^2 + \frac{1}{\xi_k}
W(\phi_k) \right] dx. 
\end{align*}
This is \reff{POGetak}. 

Finally, if $g \in C_c^1(\Omega, \R^n)$, then it follows from \reff{etakconv} 
and \reff{perimetermeasure} that 
\begin{align*}
 \lim_{k\to \infty} \int_\Omega \nabla \eta_k \cdot g \, dx 
 =
- \lim_{k \to \infty} \int_\Omega  \eta_k  \nabla \cdot  g \, dx 
 = -\int_G  \nabla \cdot  g \, dx 
 = 
-\int_{\partial^* G} g\cdot \nu \, d\calH^{n-1}. 
\end{align*}
Since $\sup_{k \ge 1} \| \eta_k \|_{W^{1,1}(\Omega)} < \infty$ by 
\reff{etakbound} and the perimeter measure 
$\| \partial G\|=\cH^{n-1}\mres(\partial^*G\cap \Omega )$ 
is a Radon measure on $\Omega$, the equation in \reff{etakweakconv} for any
function $g \in C_c(\Omega, \R^n)$ follows from the fact that such a 
function can be approximated uniformly on any compact subsets of $\Omega$ 
by functions in $C_c^1(\Omega, \R^n)$. 
\end{proof}

We denote $B(\sigma) = \cup_{i=1}^N B(x_i, \sigma)$ for any $\sigma > 0.$ 
The following is the last lemma we need to prove our $\Gamma$-convergence result: 


\begin{lemma}
\label{l:setapprox}
Let $G$ be a measurable subset of $\Omega$ such that $P_\Omega(G) < \infty,$ 
$G \supseteq B(\sigma)$ for some $\sigma > 0$, and $|G| < |\Omega|$. 
Then there exist bounded open sets $D_k \subseteq \R^3$ $(k = 1, 2, \dots)$ that satisfy the following
properties: 
\begin{compactenum}
\item[\rm (1)]
For each $k \ge 1$, $D_k \cap \Omega \supseteq B(\sigma/2); $ 
\item[\rm (2)]
For each $k \ge 1$, $\partial D_k$ is a nonempty compact hypersurface of class $C^\infty$ and 
$\partial D_k \cap \Omega$ is of class $C^2;$ 
\item[\rm (3)]
For each $k \ge 1$, $ \calH^2(\partial D_k \cap \partial \Omega) = 0; $ 
\item[\rm (4)]
$ | (D_k \cap \Omega ) \Delta G | \to 0$ as $k \to \infty;$
\item[\rm (5)]
$P_\Omega(D_k)  = P_\Omega( D_k \cap \Omega) \to P_\Omega(G)$ as $k \to \infty. $
\end{compactenum}
\end{lemma}

This lemma is similar to Lemma~1 in \cite{Modica_ARMA87} and Lemma~1
in \cite{Sternberg_ARMA88}. Here we assume $G \supseteq B(\sigma).$
Moreover, part (1) above replaces the volume constraint
$| D_k \cap \Omega | = |G|$ in \cite{Modica_ARMA87,Sternberg_ARMA88}.
An outline of the proof of this lemma is given in the proof of
Lemma~2.2 in \cite{LiZhao_SIAP13}.  For completeness, here we provide
the main steps of proof, pointing out how the property (1) is
satisfied.

\begin{proof}[Proof of Lemma~\ref{l:setapprox}]
Since $P_\Omega (G) < \infty$, there exists $u \in \mbox{BV}\, (\R^3)\cap L^\infty(\R^3)$
such that $u = \chi_G $ in $\Omega$ and
\begin{equation}
\label{u}
\int_{\partial\Omega}|\nabla u|\, d\calH^2 = 0;
\end{equation}
cf.\ Sections 2.8 and 2.16 in \cite{Giusti84}.  Since $\Omega$ is bounded, 
by using mollifiers, we can further modify $u$ so that it is compactly supported. 
Notice that $u = 1$ on $B(\sigma)$.
By using mollifiers again, we can construct $u_k \in C^\infty(\R^3)$ $(k = 1, 2, \dots)$ such that
$\mbox{supp}\,(u_k) \subseteq B(0,L)$ $(k=1,2,\dots)$ for some $L > 0$ sufficiently large, 
$u_k = 1$ in $B(\sigma/2)$ $(k = 1, 2, \dots)$, $u_k \to u $ in $L^1(\Omega)$, and
using \reff{u}
\[
\lim_{k\to \infty} \int_\Omega | \nabla u_k | \, dx  = | \nabla u |_{BV(\Omega)} = P_\Omega(A);
\]
cf.\ Sections 2.8 and 2.16 in \cite{Giusti84}.

For any $t \in \R$, we define $D_k (t) = \{ x \in \R^3: u_k(x) > t \}$ $(k = 1, 2, \dots)$.
Following Sections 1.24 and 1.26 in \cite{Giusti84}, 
and the proof of Lemma~1 in \cite{Modica_ARMA87} and
Lemma~1 in \cite{Sternberg_ARMA88} (using the co-area formula and Sard's Theorem),
there exists $t_0 \in (0,1)$ and a subsequence of $\{ D_k(t_0)\}$, 
not relabeled, that satisfy (2)--(5) in the lemma with $D_k = D_k(t_0)$ $(k = 1, 2, \dots)$.  
Clearly, for each $k \ge 1$, $D_k$ is an open set with 
 $D_k\subseteq B(0,L).$ Moreover, 
\[
D_k \supseteq \{ x\in \R^3: u_k (x) = 1 \} \supseteq B(\sigma/2), 
\qquad k = 1, 2, \dots
\]
 This, and the fact that $B(\sigma) \subseteq G \subseteq \Omega,$ implies part (1). 
\end{proof}

We are now ready to prove Theorem~\ref{t:EnergyConvergence}.

\begin{proof}[Proof of Theorem~\ref{t:EnergyConvergence}]
Fix $\xi_k \searrow 0.$ 

(1) The liminf condition. 
Assume that $\phi_k \to \phi$ in $L^1(\Omega).$
If 
$
\liminf_{k\to\infty}F_{\xi_k}[\phi_k]=+\infty,
$
then \qref{liminf-ineq} is true. Otherwise, we may assume, without loss of generality, that
\[
 \lim_{k\to\infty}F_{\xi_k}[\phi_k] = \liminf_{k\to\infty}F_{\xi_k}[\phi_k] < \infty
\] 
and that there exists a constant $C > 0$ such that $ F_{\xi_k}[\phi_k ]\leq C$ for all $k \ge 1.$ 
By the definition of functional $F_\xi$ (cf.\ \reff{newFxiphi}), 
this implies that $\phi_k \in H^1(\Omega)$ for each $k \ge 1.$ 
Hence, since $\{ F_{\xi_k}[\phi_k]\}$ is bounded, it follows from Lemma~\ref{l:LowerBound} that 
\[
\sup_{k \ge 1} \int_\Omega \left[ \frac{\xi_k}{2} |\nabla\phi_k|^2 + \frac1{\xi_k}W(\phi_k) 
\right] dx < \infty. 
\] 
Since $W(s) = 18s^2 (s-1)^2$ has exactly two minimum points $0$ and $1$, 
by a usual argument \cite{Modica_ARMA87}, there exists a subsequence of $\{ \phi_{k} \}$, 
not relabeled, 
that converges strongly in $L^1(\Omega)$ and a.e.\ in $\Omega$ to $\chi_G$ for some 
measurable subset $G\subseteq \Omega$ of finite perimeter in $\Omega$. Since 
$\phi_k \to \phi$ in $L^1(\Omega)$, we have $\phi = \chi_G$ a.e.\ in $\Omega$. 
Since $\{ F_{\xi_k}[\phi_k]\}$ is bounded, 
$\{ \| \phi_k \|_{L^4(\Omega)} \}$ is bounded by Lemma~\ref{l:LowerBound}.  Hence, it follows
from Lemma~\ref{l:Lq} that $\phi_k \to \chi_G$ in $L^q(\Omega)$ for any $q \in [1, 4).$

Since $\phi_k \to \chi_G $  in $L^2(\Omega)$, 
	\begin{align}
\label{liminf-term1}
	|G| = \int_\Omega \chi_G^2 \, dx  = \lim_{k\to\infty} \int_\Omega \phi_k^2\;dx. 
	\end{align}	
Lemma~\ref{l:etak} implies that 
	\begin{align} \label{liminf-term2}
 P_\Omega(G) \leq \liminf_{k\to\infty} \int_\Omega \left[ \frac{\xi_k}{2}|\nabla\phi_k|^2 
			+\frac1{\xi_k} W(\phi_k)\right] dx.
	\end{align}
By Fatou's Lemma, the convergence $\phi_k \to \chi_G $ a.e.\ in $\Omega$, the convergence 
$\phi_k \to \chi_G $ in $L^2(\Omega)$, and the fact that $U$ is bounded below, we obtain
\begin{align}
		 \label{liminf-term3} 
\int_{\Omega\setminus G} U \;dx 
& = \int_{\{x\in \Omega\setminus G: U(x) > 0\}} (\chi_G - 1)^2 U\, dx 
+ \int_{\{x\in \Omega\setminus G: U(x) \le  0\}} (\chi_G - 1)^2 U\, dx 
\nonumber 
\\
& \le \liminf_{k\to \infty}  \int_{\{x\in \Omega\setminus G: U(x) > 0\}} (\phi_k - 1)^2 U\, dx 
+ \lim_{k\to \infty} \int_{\{x\in \Omega\setminus G: U(x) \le  0\}} (\phi_k - 1)^2 U\, dx 
\nonumber \\
& = \liminf_{k\to \infty}  \int_{\Omega\setminus G} (\phi_k - 1)^2 U\, dx. 
\end{align}
Since $\{ \| \phi_k \|_{L^4(\Omega)} \}$ is bounded by Lemma~\ref{l:LowerBound}
and $\phi_k \to \chi_G$ in $L^1(\Omega),$ Theorem~\ref{t:PBenergy} implies that  
\begin{equation}
\label{liminf-elec}
\lim_{k\to \infty} \min_{u \in \calA} E_{\phi_k}[u] = \min_{u \in \calA} E_{\chi_G }[u]. 
\end{equation}
The liminf inequality \reff{liminf-ineq} now follows from  \reff{liminf-term1}--\reff{liminf-elec}.

\smallskip 


(2) The recovering sequence. 
Let $\phi \in L^1(\Omega)$. 
	If $F_0[\phi]=+\infty$, then we can take $\phi_k=\phi$ for all $k\ge 1$ to obtain 
\reff{limsup-ineq}. Assume $F_0[\phi]<\infty.$ We then have $\phi = \chi_G\in BV(\Omega)$
for some measurable subset $G \subseteq \Omega$ of finite perimeter in $\Omega.$   
We divide the rest of proof into two steps. 

{\it Step 1.} We first consider the case that $G = D\cap \Omega$ for
some bounded open set $D \subset \R^3$ such that the boundary
$\partial D$ is a nonempty compact hypersurface of class $C^\infty$,
$\partial D\cap \Omega $ is $C^2$, and
$\calH^2(\partial D\cap \partial \Omega) = 0$.  
It follows from a standard argument \cite{Sternberg_ARMA88, 
  Modica_ARMA87, LiZhao_SIAP13}, for $\xi_k\searrow 0$, there exist
$\phi_k\in H^1(\Omega)$ $(k=1,2,\dots)$ satisfying
\begin{align}
\label{phi_n1}
& 0\leq \phi_k \leq \chi_G  \quad \mbox{in } \Omega,  \\
\label{phi_n2}  
& \phi_k = 1 \quad \mbox{in } G_{k}:=\left\{ x\in G: 
\mbox{dist}(x,\partial G)\geq \sqrt{\xi_k} \right\}, \\
\label{phi_n3}
&\phi_k = 0 \quad\mbox{in } \Omega \setminus G, \\
\label{phi_n4} 
& \phi_k\to \chi_G \quad \mbox{strongly in } L^1(\Omega) \mbox{ and a.e.\ in }\Omega, \\	
\label{limsup-term2}
&\limsup_{k\to\infty} \int_\Omega\left[   \frac{\xi_k}{2} |\nabla\phi_k|^2  
+\frac1{\xi_k}W(\phi_k)\right] dx  
 \leq P_\Omega(G).  
\end{align}
	
By \reff{phi_n1}, \reff{phi_n4}, and Lemma~\ref{l:Lq}, we have 
$\phi_k \to \chi_G$ in $L^q(\Omega)$ for any $q > 1$. Hence
	\begin{align} \label{limsup-term1}
		\lim_{k\to\infty} \int_\Omega \phi_k^2\;dx & =  \int_\Omega \chi_G^2\;dx =  |G|.
	\end{align}
Since $F_0[\chi_G]< \infty$, by \reff{def-F} with $G$ replacing $A$,  
the integral of $U$ over $\Omega \setminus G $ is finite. 
 Since $G = D\cap \Omega$ is open and $\partial D \cap \Omega $ is $C^2$, 
it follows from our assumptions on $U$, all points $x_i \in \Omega $ 
$  (1 \le i \le N)$ must be interior points of $G.$ 
Consequently, there exists $r_0 > 0$ and $N_0 \ge 1$ such that 
$B(r_0):= \cup_{i=1}^N B(x_i, r_0) \subseteq G_{k} \subseteq G$ for all $k \ge N_0.$ 
Hence, by \reff{phi_n2}, $\phi_k = 1$ on $B(r_0)$ for all $k \ge N_0.$ 
Note that $U$ is bounded on $\Omega\setminus B(r_0).$ Therefore, 
by \reff{phi_n2} and the convergence $\phi_k \to \chi_G$ in $L^2(\Omega)$,  
\begin{align}
\label{limsup-term3}
\lim_{k\to\infty} \int_\Omega (\phi_k-1)^2 U\;dx
&= \lim_{k\to\infty}  \int_{\Omega\setminus B(r_0)} (\phi_k-1)^2 U\;dx  \nn\\
& =   \int_{\Omega\setminus B(r_0)} (\chi_G -1)^2 U\;dx 
\nonumber \\
& =   \int_{\Omega \setminus G}  U\;dx.  
\end{align}
By Theorem~\ref{t:PBenergy}, 
\begin{equation}
\label{limsup-term4}
\lim_{k\to \infty} \min_{u \in \calA}  E_{\phi_k} [u] =  \min_{u \in \calA} E_{\chi_G}[u].
\end{equation}
Combining \qref{limsup-term2}--\qref{limsup-term4}, we obtain \qref{limsup-ineq}.

\smallskip 


{\it Step 2.} We now assume that $G\subseteq \Omega$ is an arbitrary measurable subset of finite
perimeter in $\Omega$.  Since $F_0[\chi_G] $ is finite, the integral
of $U$ over $\Omega \setminus G$ is finite. This implies that
$| G | > 0.$ If $|G| = |\Omega|$ then $P_\Omega(G) = 0.$ We can thus
choose $\phi_k = \chi_G$ to get the limsup inequality
\reff{limsup-ineq}. We assume now $0 < |G| < |\Omega|.$

Choose $\sigma_k \searrow 0$ 
such that the closure of $ B(\sigma_k) := \cup_{i=1}^N B(x_i,\sigma_k) $ is included in $\Omega,$
$U \ge 0$ on $B(\sigma_k)$, and $0 < |G \cup B(\sigma_k)| < |\Omega|$  for each $k \ge 1.$  
Denote $\widehat{G}_k = G \cup B(\sigma_k) $ for $k \ge 1$.  
Then $G \subseteq \widehat{G}_{k+1} \subseteq \widehat{G}_k$ for all $k \ge 1$ and $\chi_{\widehat{G}_k} 
\to \chi_G$ in $L^1(\Omega)$. We claim that 
\begin{equation}
\label{GkGclaim}
\limsup_{k\to \infty} F_0[\chi_{\widehat{G}_k}] \le F_0[\chi_G]. 
\end{equation}
Clearly, 
\begin{equation}
\label{GkG}
|\widehat{G}_k | = |G| + |B (\sigma_k) \setminus G|  \to |G|  \quad \mbox{as } k \to \infty. 
\end{equation}
Moreover \cite{Giusti84}, 
\begin{align}
\label{PGkPG}
\limsup_{k \to \infty} P_{\Omega} ( \widehat{G}_k) 
 &= \limsup_{k \to \infty} P_{\Omega} ( G \cup B_k)
\nn \\
& \le  \limsup_{k \to \infty} \left[ P_{\Omega}(G) + P_{\Omega} (B_k) \right]
\nonumber \\
 & =  P_{\Omega}(G) + \lim_{k\to \infty} P_{\Omega}  (B_k)
\nn \\
&  =  P_{\Omega}(G).
\end{align}
Since $\Omega \setminus \widehat{G}_k \subseteq \Omega \setminus \widehat{G}_{k+1}$, 
we have by the Lebesgue Monotone Convergence Theorem that  
\begin{align*}
\lim_{k\to \infty} \int_{\Omega \setminus \widehat{G}_k} \chi_{\{ x \in \Omega: U(x) > 0 \}} U \, dx 
& = \lim_{k\to \infty} \int_\Omega \chi_{\Omega \setminus \widehat{G}_k} 
    \chi_{\{ x \in \Omega: U(x) > 0 \}} U \, dx 
\\
& = \int_\Omega \chi_{\Omega \setminus G} \chi_{\{ x \in \Omega: U(x) > 0 \}} U \, dx 
\\
& =  \int_{\Omega \setminus G} \chi_{\{ x \in \Omega: U(x) > 0 \}} U \, dx. 
\end{align*}
Since $U$ is bounded below and $| \Omega \setminus \widehat{G}_k | \to | \Omega \setminus G|,$ 
\[
\lim_{k\to \infty} \int_{\Omega \setminus \widehat{G}_k} \chi_{\{ x \in \Omega: U(x) \le  0 \}} U \, dx 
= \int_{\Omega \setminus G} \chi_{\{ x \in \Omega: U(x) \le  0 \}} U \, dx. 
\]
Combining the above two equations, we get 
\begin{equation}
\label{GkGU}
\lim_{k \to \infty} \int_{\Omega \setminus \widehat{G}_k } U\, dx
= \int_{\Omega \setminus G} U\, dx. 
\end{equation}
By Theorem~\ref{t:PBenergy}, 
\begin{equation}
\label{GkGelec}
\lim_{k \to \infty} \min_{u \in \calA} E_{\chi_{\widehat{G}_k}}[u] =  \min_{u \in \calA} E_{\chi_G}[u].  
\end{equation}
Now, \reff{GkGclaim} follows from \reff{GkG}--\reff{GkGelec}. 

Fix an arbitrary $k \ge 1$. It follows from Lemma~\ref{l:setapprox} that 
there exist open sets $D_{k,j} \subseteq \R^3 $ $(j = 1, 2, \dots)$ such that, 
for each $j \ge 1$ and $ G_{k,j} :=  D_{k,j} \cap \Omega,$ $G_{k,j} \supseteq B(\sigma_k/2),$
 $\partial D_{k,j} $ is $C^\infty$ and $\partial D_{k,j} \cap \Omega $ is $C^2,$  
and $\calH^2(\partial D_{k,j} \cap  \partial \Omega) = 0,$ and that $|G_{k,j} \Delta \widehat{G}_k | \to 0$, 
which is equivalent to $\chi_{G_{k,j}} \to \chi_{\widehat{G}_k}$ in $L^1(\Omega)$,  and 
$P_\Omega(G_{k,j})\to P_\Omega (\widehat{G}_k)$ as $j \to \infty.$ 
Clearly, $|G_{k,j}| \to |\widehat{G}_k|$ as $j \to \infty.$ 
Since each $G_{k,j}\supseteq B(\sigma_k/2)$ and $\chi_{G_{k,j}} \to \chi_{\widehat{G}_k}$ in $L^1(\Omega)$,  
\[
\lim_{j\to \infty} \int_{\Omega\setminus G_{k,j}} U\, dx = \int_{\Omega\setminus \widehat{G}_{k}} U\, dx. 
\]
By Theorem~\ref{t:PBenergy}, $\min_{u \in \calA} E_{\chi_{G_{k,j}}}[u]
\to \min_{u \in \calA} E_{\chi_{\widehat{G}_k}} [u]$ as $j \to \infty.$  Therefore, 
\[
\lim_{j\to \infty} F_0[\chi_{G_{k,j}}] = F_0[\chi_{\widehat{G}_k}], \quad k = 1, 2, \dots
\]

By induction, we can choose $j_1 < j_2 < \cdots$ with $j_k \to \infty$
such that, with the notation $H_k =  G_{k, j_k}$ for all $k \ge 1$, 
\[
\| \chi_{ H_k} - \chi_{\widehat{G}_k} \|_{L^1(\Omega)} < \frac{1}{k} \quad \mbox{and} \quad 
|F_0[\chi_{H_k} ] - F_0[\chi_{\widehat{G}_k}] | < \frac{1}{k}, \quad k = 1, 2, \dots
\]
These, together with the fact that $\chi_{\widehat{G}_k} \to \chi_G$ 
in $L^1(\Omega)$ and \reff{GkGclaim}, imply that 
\begin{equation}
\label{Hk}
\lim_{k\to \infty} \| \chi_{ H_k} - \chi_{G} \|_{L^1(\Omega)} = 0
\quad \mbox{and} \quad 
\limsup_{k\to \infty} F_0[\chi_{H_{k} }] \le F_0[\chi_G]. 
\end{equation}
By Step 1, we can find for each $k \ge 1$ a recovering sequence $ \{ \phi_{k,l} \}_{l=1}^\infty$  
for $\chi_{H_k} $ such that all $\phi_{k,l} \in H^1(\Omega)$ $(l = 1, 2, \dots)$, 
\begin{equation}
\label{phikl}
\lim_{l\to \infty} \| \phi_{k,l} - \chi_{H_k}\|_{L^1(\Omega)} = 0
\quad \mbox{and} \quad 
\limsup_{l\to \infty} F_{\xi_l}  [\phi_{k,l} ] \le F_0[\chi_{H_k} ], 
\quad k = 1, 2, \dots
\end{equation}
By \reff{Hk} and \reff{phikl}, and induction, we 
can choose $l_1 < l_2 < \cdots $ with $l_k \to \infty$ such that 
$\phi_{k,l_k} \to \chi_{G}$ in $L^1(\Omega)$ and 
\[
\limsup_{k\to \infty} F_{\xi_{l_k} }  [\phi_{k,l_k} ] \le F_0[\chi_G ].  
\]
The proof is complete.
\end{proof}


\begin{proof}[Proof of Corollary~\ref{c:existenceF0}]
Let $\xi_k \searrow 0$. For each $k \ge 1$, let $\phi_k \in H^1(\Omega)$
be such that $F_{\xi_k}[\phi_k] = \min_{\phi \in L^1(\Omega)}F_{\xi_k}[\phi]$; cf.\
Lemma~\ref{l:minimizerFxi}. By Lemma~\ref{l:LowerBound} and 
comparing $F_{\xi_k}[\phi_k]$ to the free energy of
the constant function $\phi = 1$, the sequence $\{ F_{\xi_k}[\phi_k]\}$ is bounded.  
Hence the corresponding sequence of the van der Waals--Cahn--Hilliard 
functionals of $\phi_k$ is also bounded. This and a usual argument
\cite{Modica_ARMA87,Sternberg_ARMA88} imply that there exists a subsequence of
$\{ \phi_k \}$, not relabeled, such that 
$\phi_k \to \chi_G$ in $L^1(\Omega)$ for some measurable subset $G$ of $\Omega$. 
Theorem~\ref{t:EnergyConvergence} then implies $\chi_G$ minimizes $F_0$. 
\end{proof}

We need the following elementary result in the proof of Theorem~\ref{t:individual}: 

\begin{lemma}
\label{l:AkBk}
Let $a_k$ and  $b_k$ $(k = 1, 2, \dots)$, and $a$ and $b$ be all nonnegative numbers such that
\[
\lim_{k\to \infty} (a_k+b_k) = a + b, \quad \liminf_{k\to \infty} a_k \ge a, 
\quad \mbox{and} \quad \liminf_{k\to \infty} b_k \ge b. 
\]
Then
\[
\lim_{k\to \infty} a_k = a \quad \mbox{and} \quad \lim_{k\to \infty} b_k = b. 
\]
\end{lemma}

\begin{proof}
Since $a_k \ge 0$ and $b_k \ge 0$ $(k = 1, 2, \dots)$ and $\{ a_k + b_k \}$ converges,
both $\{ a_k \}$ and $\{ b_k \} $ are bounded. 
Let $\{ a_{k_j} \}$ be any subsequence of $\{ a_k \}.$ 
Let $\{ a_{{k_j}_i}\}$ be a further subsequence such that 
\begin{equation}
\label{kji}
\lim_{i \to \infty} a_{k_{j_i}} = \liminf_{j \to \infty} a_{k_j}. 
\end{equation}
We have then
\[
a + b = \liminf_{j\to \infty} (a_{k_j} + b_{k_j}) 
\ge \liminf_{j \to \infty} a_{k_j} + \liminf_{j \to \infty} b_{k_j} 
\ge \liminf_{k \to \infty} a_{k} + \liminf_{k \to \infty} b_{k} \ge a + b,  
\]
leading to 
\[
0 \ge \left(  \liminf_{j \to \infty}  a_{k_j} - a \right) 
  + \left( \liminf_{j \to \infty} b_{k_j} - b \right) \ge  0.  
\]
Each term in the sum is nonnegative, and hence is $0$. Thus $ \liminf_{j \to \infty}  a_{k_j} =  a.  $
This and \reff{kji} imply that $a_{k_{j_i}} \to a$ as $i \to \infty$, and hence
$a_k \to a$ as $k\to \infty$. Similarly, $b_k \to b$ as $k\to \infty$. 
\end{proof}

We are now ready to prove Theorem~\ref{t:individual}.  

\begin{proof}[Proof of Theorem~\ref{t:individual}]
Since $\{ F_{\xi_k}[\phi_k]\}$ converges, it is bounded. Lemma~\ref{l:LowerBound} then 
implies that $\sup_{k \ge 1}\|\phi_k \|_{L^4(\Omega)} < \infty.$ 
Since $\phi_k \to \chi_G$ a.e.\ in $\Omega$, Lemma~\ref{l:Lq} implies that 
$\phi_k \to \chi_G$ in $L^q(\Omega)$ for any $q\in [1, 4).$ 
Hence, \reff{volume} follows. Moreover, Theorem~\ref{t:PBenergy} implies \reff{Ele}. 

By our assumptions on $U$ and the Lebesgue Dominated Convergence Theorem, 
\begin{equation}
\label{Une0}
\lim_{k \to \infty} \int_{\{ x\in \Omega: U(x) \le 0 \}} (\phi_k - 1)^2 U \, dx 
=  \int_{\{ x\in \Omega: U(x) \le 0 \}} \chi_{\Omega \setminus G}  U \, dx.  
\end{equation}
Since $F_{\xi_k}[\phi_k] \to F_0[\chi_G]$ with $F_0[\chi_G]$ being finite,    
it follows from \reff{volume}, \reff{Ele}, and \reff{Une0} that  
\begin{align}
\label{ABge0}
&\lim_{k\to \infty} \left\{  \gamma_0 \int_\Omega \left[ \frac{\xi_k}{2} |\nabla \phi_k|^2 
+ \frac{1}{\xi_k} W(\phi_k) \right] dx
+ \rho_0 \int_{\{ x\in \Omega: U(x) > 0 \}} (\phi_k-1)^2 U\, dx \right\}
\nonumber 
\\
&\qquad 
= \lim_{k \to \infty} \left\{ F_{\xi_k}[\phi_k] - P_0 \int_\Omega \phi_k^2\, dx 
- \rho_0  \int_{\{ x\in \Omega: U(x) \le  0 \}} (\phi_k-1)^2 U\, dx
+ \min_{u \in \calA} E_{\phi_k}[u] \right\}
\nonumber 
\\
&\qquad 
= F_{0} [\chi_G]  - P_0 |G| - \rho_0  \int_{\{ x\in \Omega: U(x) \le  0 \}} 
\chi_{\Omega\setminus G} U\, dx + \min_{u \in \calA} E_{\chi_G} [u] 
\nonumber 
\\
&\qquad 
=  \gamma_0 P_\Omega(G) +  \rho_0 \int_{\{ x\in \Omega: U(x) > 0 \}} 
\chi_{\Omega \setminus G } U\, dx. 
\end{align}
By Lemma~\ref{l:etak}, we have 
\begin{equation}
\label{Age0}
\liminf_{k\to \infty} 
 \int_\Omega \left[ \frac{\xi_k}{2} |\nabla \phi_k|^2 
+ \frac{1}{\xi_k} W(\phi_k) \right] dx \ge P_\Omega(G).   
\end{equation}
Fatou's Lemma implies that 
\begin{equation}
\label{Bge0}
\liminf_{k \to \infty} \int_{\{ x\in \Omega: U(x) > 0 \}} (\phi_k - 1)^2 U \, dx 
\ge \int_{\{ x\in \Omega: U(x) > 0 \}} \chi_{\Omega\setminus G}  U \, dx.  
\end{equation}
By \reff{ABge0}--\reff{Bge0} and Lemma~\ref{l:AkBk}, 
the inequalities \reff{Age0} and \reff{Bge0} become equalities. 
Therefore \reff{surface} is true; and further, \reff{LJ} is true. 

Finally, since all $F_0[\chi_G]$, $|G|$, $P_\Omega(G)$, and $F_{\rm ele}[G]$ are finite, 
the right-hand side of \reff{LJ} is also finite. 
\end{proof}

\section{Force Convergence}
\label{s:ForceConvergenceSolvation}


We first prove Theorem~\ref{t:ForceConvSolvation}. We then focus on the
proof of Theorem~\ref{th:CH-force-conv}, which is for a general space dimension $n \ge 2.$ 

\begin{proof}[Proof of Theorem~\ref{t:ForceConvSolvation}]
Since $F_{\xi_k}[\phi_k]\to F_0[\chi_G]$, Lemma~\ref{l:LowerBound} implies that 
$ \{ \| \phi_k \|_{L^4(\Omega)} \}$ is bounded. Since, $\phi_k \to \chi_G$ a.e.\ 
in $\Omega$, Lemma~\ref{l:Lq} then implies that $\phi_k \to \chi_G$ in $L^q(\Omega)$
for any $q \in [1, 4).$ This implies \reff{f_vol-conv}; it also implies 
\reff{f_vdW-conv} as both $U$ and $\nabla U$ are continuous on $\mbox{supp}\,(V).$
The second equation \reff{CH-force-conv} is part of Theorem~\ref{th:CH-force-conv}.  
Finally, the equation \reff{f_ele-conv} is part of Theorem~\ref{th:f_ele-conv}. 
\end{proof}

To prove Theorem~\ref{th:CH-force-conv}, we need the following lemma which states that the 
convergence of phase-field surface energies to their sharp-interface limit 
implies the asymptotic equi-partition of energies.  Indeed, we prove that 
\begin{equation*}
 	\frac{\xi_k}{2} |\nabla \phi_k|^2 -  \frac{1}{\xi_k} W(\phi_k) 
\to 0 \quad\mbox{strongly in }L^1(\Omega) 	\mbox{ as }k\to \infty. 
\end{equation*}
This is stronger than the weak convergence of the discrepancy measures 
\[
\left[ \frac{\xi_k}{2} |\nabla \phi_k|^2 -  \frac{1}{\xi_k} W(\phi_k)\right] dx
\quad (k = 1, 2, \dots) 
\]
that are defined in \cite{RogerSchatzle06,Ilmanen1993}:   

\begin{lemma}[Asymptotic equi-partition of energy]
\label{l:equienergy}
Let $\xi_k \searrow 0,$ $\phi_k\in H^1(\Omega)$ $(k = 1, 2, \dots),$ 
and $G\subseteq \Omega$ be measurable with $P_\Omega(G) < \infty.$ 
Assume that $\phi_k \to \chi_G$ a.e.\ in $\Omega.$ Assume also that 
\begin{equation}
\label{POG}
\lim_{k\to \infty} \int_\Omega 
\left[ \frac{\xi_k}{2} |\nabla \phi_k|^2 + \frac{1}{\xi_k} W(\phi_k) \right] dx = P_\Omega(G).   
\end{equation}
Then, we have 
\begin{equation}
\label{equienergy-0}
\lim_{k \to \infty} \int_\Omega \left| \sqrt{ \frac{\xi_k}{2} } |\nabla \phi_k| 
- \sqrt{ \frac{W(\phi_k)}{\xi_k}} \right|^2 dx = 0,  
\end{equation}
and
\begin{equation} 
\label{equienergy}
 \lim_{k\to \infty} \int_\Omega \left|  \frac{\xi_k}{2} |\nabla \phi_k|^2 
-  \frac{1}{\xi_k} W(\phi_k) \right| dx = 0. 	
\end{equation}
\end{lemma}

\begin{proof}
Define $\eta_k = \eta_k(x)$ as in Lemma~\ref{l:etak}. 
We have by Lemma~\ref{l:etak} and \reff{POG} that 
\begin{align*}
0 & \le
 \limsup_{k \to \infty} \int_\Omega 
\left| \sqrt{ \frac{\xi_k}{2} } |\nabla \phi_k| - \sqrt{ \frac{W(\phi_k)}{\xi_k}} \right|^2 dx
\\
& =  
 \limsup_{k \to \infty} \int_\Omega \left[
\frac{\xi_k}{2}  |\nabla \phi_k|^2 + \frac{1}{\xi_k} W(\phi_k) 
 -  \sqrt{ 2 W(\phi_k) } |\nabla \phi_k | \right]    dx
\\
& =  \lim_{k \to \infty} \int_\Omega \left[ 
\frac{\xi_k}{2}  |\nabla \phi_k|^2 + \frac{1}{\xi_k} W(\phi_k) \right] 
- \liminf_{k \to \infty} \int_\Omega  \sqrt{ 2 W(\phi_k ) } |\nabla \phi_k |\,  dx 
\\
& = P_\Omega(G) - \liminf_{k \to \infty} \int_\Omega  |\nabla \eta_k |\, dx 
\\
& \le 0. 
\end{align*}
This proves \reff{equienergy-0}. 
By \reff{POG} and \reff{equienergy-0}, we have 
\begin{align*}
& 
\int_\Omega \left|  \frac{\xi_k}{2} |\nabla \phi_k|^2 -  \frac{1}{\xi_k} W(\phi_k) \right| dx 
\\
&\qquad 
= \int_\Omega 
\left| \sqrt{ \frac{\xi_k}{2} } |\nabla \phi_k| - \sqrt{ \frac{W(\phi_k)}{\xi_k}} \right| 
\, 
\left| \sqrt{ \frac{\xi_k}{2} } |\nabla \phi_k| + \sqrt{ \frac{W(\phi_k)}{\xi_k}} \right|\, dx
\\
&\qquad 
\le \left( \int_{\Omega} \left| \sqrt{ \frac{\xi_k}{2} } |\nabla \phi_k| 
- \sqrt{ \frac{W(\phi_k)}{\xi_k}} \right|^2 dx \right)^{1/2}   \left( 2 \int_\Omega 
\left[ \frac{\xi_k}{2} |\nabla \phi_k|^2 + \frac{1}{\xi_k} W(\phi_k) \right] dx \right)^{1/2}   
\\
&\qquad 
\to 0 \qquad \mbox{as } k \to \infty,    
\end{align*}
implying \reff{equienergy}. 
\end{proof}

We are now ready to prove Theorem~\ref{th:CH-force-conv}.  

\begin{proof}[Proof of Theorem~\ref{th:CH-force-conv}.]
Suppose \reff{CH-force-conv1} is true for any $\Psi \in C_c(\Omega, \R^n \times \R^n).$ 
Let $V \in C_c^1(\Omega, \R^n).$ 
Under the additional assumptions on $\phi_k$ $(k \ge 1)$ and $G$, we have 
by \reff{weak-f_sur} in Lemma~\ref{l:Stress}, 
\reff{CH-force-conv1} with $\Psi = \nabla V$, and \reff{weak-f0_sur} 
in Lemma~\ref{l:sharpboundaryforce} that 
\begin{align*}
& \lim_{k\to\infty}\int_\Omega\left[  -\xi_k\Delta\phi_k 
+ \frac1{\xi_k} W'(\phi_k)\right] \nabla\phi_k \cdot V\;dx
\\
& \qquad =  - \lim_{k\to\infty}  \int_\Omega T_{\xi_k}(\phi_k) : \nabla V\, dx 
\\
&\qquad 
= - \int_{\partial G} (I - \nu \otimes \nu ):  \nabla V \, d \calH^{n-1}    
\\
&\qquad 
= - (n-1)\int_{\partial G} H\nu\cdot V \, d S, 
\end{align*}
proving \reff{CH-force-conv2}. 

We now  prove \reff{CH-force-conv1}. We claim that it suffices to prove that 
\begin{align}
\label{tensor-conv2}
\lim_{k\to \infty} \int_\Omega \xi_k\nabla\phi_k\otimes\nabla\phi_k:\Psi\;dx 
= \int_{\partial^*G}\nu\otimes\nu:\Psi\;d\cH^{n-1}
\qquad \forall \Psi\in C_c(\Omega;\R^{n\times n}). 
\end{align}
In fact, suppose \reff{tensor-conv2} is proved. Notice for any $a \in \R^n$, 
$|a|^2 = a\otimes a : I$. Let $\Psi \in C_c(\Omega, \R^{n\times n})$.  
Then $(I:\Psi) I\in C_c(\Omega, \R^{n\times n})$.  
Hence, it follows from Lemma~\ref{l:equienergy} and 
\reff{tensor-conv2}, with $(I:\Psi)I $ replacing $\Psi$, that 
\begin{align*}
&\lim_{k\to\infty}\int_\Omega \left[ \frac{\xi_k}{2}|\nabla\phi_k|^2 
         + \frac1{\xi_k}W(\phi_k)\right] I : \Psi\, dx 
\\
& \qquad = \lim_{k\to\infty}\int_\Omega  \xi_k|\nabla\phi_k|^2 I: \Psi\,  dx
\\
& \qquad = \lim_{k\to\infty}\int_\Omega  \xi_k \nabla \phi_k \otimes \nabla \phi_k  : (I: \Psi) I  \, dx
\\
&\qquad  =  \int_{\partial^*G}\nu\otimes\nu : (I:\Psi) I \, d\cH^{n-1}
\\
&\qquad 
=  \int_{\partial^*G}I:\Psi \, d\cH^{n-1}. 
\end{align*}
This, togehter with \reff{tensor-conv2}, implies \reff{CH-force-conv1}. 

It remains to prove \reff{tensor-conv2}. 
Fix $\Psi \in C_c(\Omega, \R^n \times \R^n)$ and let $\sigma > 0.$ 
Recall that 
the reduced boundary $\partial^*G$ has the decomposition 
\cite{Giusti84,Ziemer_Book89,EvansGariepy_Book92}
\[
\partial^*G = \biggl( \bigcup_{j=1}^\infty K_j \biggr) \bigcup Q, 
\]
where $K_j$ $(j = 1, 2, \dots)$ are disjoint compact sets, each being
a subset of a $C^1$-hypersurface $S_j \subset \Omega$, 
and $Q\subset \partial G$ with $\| \partial G\|(Q) = 0.$ 
The vector $\nu(x) $ at some $x \in K_j$ for some $j$ is the normal to $S_j$. 
Moreover,
\begin{equation}
\label{GK}
\sum_{j=1}^\infty \calH^{n-1}(K_j) = \calH^{n-1}(\partial^* G) = 
\|\partial G \|(\Omega) = P_\Omega(G) < \infty. 
\end{equation}
Let  $J $ be large enough so that
\begin{equation}
\label{Jsigma}
\sum_{j=J+1}^\infty \calH^{n-1}(K_j) < \sigma. 
\end{equation}
Since $K_j$ $(j = 1, \dots, J)$ are disjoint, there exist disjoint open 
sets $U_j \subset \overline{U}_j \subset \Omega$
such that $K_j \subset U_j$ $(j = 1, \dots, J).$  
For each $j$ $(1 \le j \le J)$, we define $d_j: U_j \to \R$ to be 
the signed distance to $S_j$ for which the sign is chosen
so that $\nu (x) = \nabla d_j(x)$ if $x \in K_j;$ 
and extend $d_j $ to $\Omega$ by setting $d_j = 0$ on $\Omega \setminus U_j.$ 
We also choose $\zeta_j \in C_c^1(\Omega)$ be such that 
$0 \le \zeta_j \le 1$ on $\Omega$, $\zeta_j = 1$ in a neighborhood of $K_j$, 
$\mbox{supp}\, (\zeta_j) \subset U_j$, and $ \zeta_j  \nabla d_j  \in C_c(\Omega, \R^n)$. 
Define $\nu_J: \Omega \to \R^n$ by
\[
\nu_J = \sum_{j=1}^J \zeta_j  \nabla d_j. 
\]
Note that $\nu_j \in C_c (\Omega, \R^n)$, $|\nu_j| \le 1$ on $\Omega$, 
and $\nu_j = \nu$ on each $K_j$ $(1 \le j \le J).$ 

We rewrite $ \xi_k\nabla\phi_k\otimes\nabla\phi_k $ as 
\begin{align}
\label{decomposition}
\xi_k\nabla\phi_k\otimes\nabla\phi_k = & \left( \sqrt{\xi_k}\nabla\phi_k 
+\sqrt{\xi_k} | \nabla \phi_k | \nu_J \right)\otimes\sqrt{\xi_k}\nabla\phi_k \nn\\
&\quad + \left( \sqrt{\frac{2W(\phi_k)}{\xi_k}} - \sqrt{\xi_k}|\nabla\phi_k|\right)
\nu_J\otimes\sqrt{\xi_k}\nabla\phi_k 
\nn\\
&\quad  - \nu_J\otimes\sqrt{2W(\phi_k)}\nabla\phi_k.
\end{align}
We claim: 
\begin{enumerate}
	\item[(1)]  $\disp\limsup_{k\to\infty}\int_\Omega\left|\sqrt{\xi_k}\nabla\phi_k 
		+ \sqrt{\xi_k} | \nabla \phi_k | \nu_J \right|^2 dx \le 4\sigma;$
        \item[(2)] $\disp \sup_{k\ge1}
\left\| \sqrt{\xi_k}\nabla\phi_k \right\|_{L^2(\Omega)} < \infty; $ 
	\item[(3)]  $\disp\lim_{k\to\infty}\int_\Omega \left[ \sqrt{\xi_k}|\nabla\phi_k| 
		-\sqrt{\frac{2W(\phi_k)}{\xi_k}}\,\right]^2dx =0;$
	\item[(4)] 
$\disp \lim_{k\to \infty} \int_\Omega \nu_J\otimes
              \sqrt{2W(\phi_k)}\nabla\phi_k:\Psi\;dx 
= - \int_{\partial^*G}  \nu_J \otimes \nu: \Psi\; d\calH^{n-1}.  $ 
\end{enumerate}
If all these claims are true, then it follows from \reff{decomposition} and \reff{Jsigma} that 
\begin{align*}
	&\limsup_{k\to\infty} \left|  \int_\Omega \xi_k\nabla\phi_k\otimes\nabla\phi_k :\Psi\;dx 
		- \int_{\partial^*G}\nu\otimes\nu:\Psi\;d\cH^{n-1} \right| 
	\\
	&\qquad \le \limsup_{k\to\infty}
\int_\Omega \left| \left( \sqrt{\xi_k}\nabla\phi_k 
	+ \sqrt{\xi_k} | \nabla \phi_k | \nu_J \right)\otimes\sqrt{\xi_k}
     \nabla\phi_k:\Psi \right| \, dx \nn\\
	&\qquad \quad + \limsup_{k\to\infty} \int_\Omega\left|\left(  \sqrt{\frac{2W(\phi_k)}{\xi_k}} 
		- \sqrt{\xi_k}|\nabla\phi_k|\right)\nu_J\otimes\sqrt{\xi_k}
		\nabla\phi_k : \Psi \right| \;dx \nn\\
	&\qquad\quad + \left| \lim_{k\to\infty}\int_\Omega\nu_J\otimes
              \sqrt{2W(\phi_k)}\nabla\phi_k:\Psi\;dx 
		+\int_{\partial^*G}\nu\otimes\nu:\Psi\;d\cH^{n-1} \right| \\
	& \qquad \le \limsup_{k\to \infty } \left[ \int_\Omega  \left| \sqrt{\xi_k}\nabla\phi_k 
	+ \sqrt{\xi_k} | \nabla \phi_k | \nu_J \right|^2  dx \right]^{1/2}
    \left(\sup_{k\ge 1}\left\| \sqrt{\xi_k}\nabla\phi_k 
          \right\|_{L^2(\Omega)}\right) \|\Psi\|_{L^\infty(\Omega)}  
\\
& \qquad \quad + \limsup_{k\to\infty}\left[ \int_\Omega \left(\sqrt{\xi_k}|\nabla\phi_k| 
		-\sqrt{\frac{2W(\phi_k)}{\xi_k}}\right)^2dx\right]^{1/2}
\\
& \qquad \qquad \cdot
	\left( \sup_{k\ge1}\|\sqrt{\xi_k}\nabla\phi_k\|_{L^2(\Omega)} \right) 
           \|  \Psi  \|_{L^\infty(\Omega)}
\\
& \qquad \qquad + \left|\int_{\partial^*G} ( \nu_J -\nu)\otimes\nu : \Psi \;d\cH^{n-1} \right| 
 \\
	& \qquad \le \sqrt{4\sigma}\left(\sup_{k\ge 1}\left\| \sqrt{\xi_k}\nabla\phi_k 
           \right\|_{L^2(\Omega)}\right) \|\Psi\|_{L^\infty(\Omega)} 
		+ 2 \|\Psi\|_{L^\infty(\Omega)} \sum_{j=J+1}^\infty \cH^{n-1}(K_j) \\
	& \qquad \le \sqrt{4\sigma}\left(\sup_{k\ge 1}\left\| \sqrt{\xi_k}\nabla\phi_k 
           \right\|_{L^2(\Omega)}\right) \|\Psi\|_{L^\infty(\Omega)}  
		+ 2\sigma  \|\Psi\|_{L^\infty(\Omega)}. 
\end{align*}
Since $\sigma > 0$ is arbitrary, this proves \qref{tensor-conv2}.

We now prove all of our claims. 
Claim (2) follows from the assumption \reff{important} of the energy convergence 
and the assumption that $P_\Omega(G) < \infty.$ 
Claim (3) is \qref{equienergy-0} in Lemma~\ref{l:equienergy}. 
Claim (4) follows from \reff{etakweakconv} in Lemma~\ref{l:etak}, 
which implies that for any $j \in \{ 1, \dots, n \}$ 
\[
\lim_{k\to \infty} \int_\Omega \partial_{x_j} \eta_k h \, dx
= -\int_{\partial^* G} \nu_j h \, d\calH^{n-1} \qquad \forall h \in C_c(\Omega),  
\]
where $\nabla \eta_k = \sqrt{2 W(\phi_k)} \nabla \phi_k$.  


Proof of Claim (1).
Noting that $|\nu_J| \le 1$, we have  for each $k \ge 1$ that 
\begin{align}
\label{Claim1}
&\frac12 \int_\Omega\left|\sqrt{\xi_k}\nabla\phi_k + \sqrt{\xi_k} 
| \nabla \phi_k | \nu_J \right|^2 dx 
\nn\\
&\qquad 	
= \frac12 \int_\Omega \left(\xi_k |\nabla\phi_k|^2  + \xi_k |\nabla\phi_k|^2 | \nu_J|^2 
		+2\xi_k|\nabla\phi_k|\nabla\phi_k
		\cdot \nu_J\right)dx 
\nn\\
&\qquad 
\le  \int_\Omega\left(\xi_k |\nabla\phi_k|^2 + \xi_k|\nabla\phi_k|\nabla\phi_k\cdot \nu_J \right)dx 
\nn\\
&\qquad 	
= \int_\Omega \left[ \frac{\xi_k}{2}|\nabla\phi_k|^2 + \frac{W(\phi_k)}{\xi_k}\right]  dx 
	+  \int_\Omega  \left[ \frac{\xi_k}{2}|\nabla\phi_k|^2 - \frac{W(\phi_k)}{\xi_k}\right] dx 
\nn\\
&\qquad \quad 
+ \int_\Omega \left[  \sqrt{\xi_k}|\nabla\phi_k|-\sqrt{\frac{2W(\phi_k)}{\xi_k}}\right] 
\sqrt{\xi_k}\nabla\phi_k \cdot\nu_J\;dx 
\nn\\
&\qquad \quad 
 +  \int_\Omega \sqrt{2W(\phi_k)}\nabla\phi_k\cdot\nu_J\;dx
\nn \\
&\qquad
	=: I_1(k) +I_2 (k) +I_3(k)  + I_4(k).	
\end{align}
By \qref{POG},
\[
\lim_{k\to \infty} I_1(k) = P_\Omega(G). 
\]
By Lemma~\ref{l:equienergy} on the asymptotic equi-partition of energy, 
\[
\lim_{k\to \infty} I_2(k) = 0. 
\]
By Claim (2) and Claim (3), 
\[
\lim_{k\to \infty} I_3(k) = 0. 
\]
By \reff{etakweakconv} in Lemma~\ref{l:etak}, 
\begin{align*}
\lim_{k\to\infty} I_4 &= - \int_{\partial^* G} \nu\cdot \nu_J\;d\cH^{n-1} 
\\
&=  - \sum_{j=1}^J\cH^{n-1}(K_j) -  \sum_{j=J+1}^\infty\int_{K_j}\nu\cdot\nu_J\;d\cH^{n-1}.
\end{align*}
Therefore, continuing from \reff{Claim1}, we have by \reff{GK}, \reff{Jsigma}, 
and the fact that $|\nu\cdot \nu_J|\le 1$ that 
\begin{align*}
&\limsup_{k\to \infty}  
\frac12 \int_\Omega\left|\sqrt{\xi_k}\nabla\phi_k + \sqrt{\xi_k} 
| \nabla \phi_k | \nu_J \right|^2 dx 
\\
&\qquad 
\le P_\Omega(G)
- \sum_{j=1}^J\cH^{n-1}(K_j) -  \sum_{j=J+1}^\infty\int_{K_j}\nu\cdot \nu_J\;d\cH^{n-1}
\\
&\qquad 
=   \sum_{j=J+1}^\infty \cH^{n-1}(K_j) -\sum_{j=J+1}^\infty\int_{K_j}
\nu\cdot\nu_J\,d\cH^{n-1}
\nn\\
&\qquad \leq 2\sum_{j=J+1}^\infty \cH^{n-1}(K_j) \nn \\
&\qquad \leq 2\sigma,  
\end{align*}
proving Claim (1). 
The proof is complete. 
\end{proof}

}

\medskip

\noindent{\bf Acknowledgments.}
This work was supported by the US National Science Foundation (NSF)
through the grant DMS-1411438 (S.D.), DMS-1319731 (B.L.), and
DMS-1454939 (J.L.).

\medskip

\bibliography{GammaSolvation}
\bibliographystyle{plain}
\end{document}